\documentclass[11pt]{amsart}

\usepackage{amsmath,amssymb}
\usepackage{hyperref}
\usepackage[active]{srcltx}
\usepackage[mathscr]{euscript}
\usepackage{enumitem}

\numberwithin{equation}{section}

\theoremstyle{plain}

\newtheorem{Thm}{Theorem}[section]
\newtheorem{Prop}[Thm]{Proposition}
\newtheorem{Lemma}[Thm]{Lemma}

\newtheorem{Cor}[Thm]{Corollary}

\theoremstyle{definition}
\newtheorem{Example}[Thm]{Example}
\newtheorem{Defi}[Thm]{Definition}
\newtheorem{Remark}[Thm]{Remark}

\newcommand\malcev{\mathop{\raise1pt\hbox{\footnotesize$\bigcirc$\kern-8pt\raise1pt
      \hbox{\tiny$m$}\kern1pt}}}

\def\pv#1{\ensuremath{{\mathsf{#1}}}}
\def\Om#1#2{\ensuremath{\overline\Omega_{#1}{\pv{#2}}}}
\def\om#1#2{\ensuremath{\Omega_{#1}{\pv{#2}}}}

\def\ZZ{\ensuremath{\mathbb{Z}}}

\def\te#1{\mathsf t_{#1}}
\def\be#1{\mathsf i_{#1}}
\def\op{\lbrack\hskip-1.6pt\lbrack\relax}
\def\cl{\rbrack\hskip-1.6pt\rbrack\relax}

\newcommand\J{\mathcal{J}}
\newcommand\R{\mathcal{R}}
\renewcommand\L{\mathcal{L}}

\newcommand\K{\mathcal{K}}

\def\Lo#1{\ensuremath{\mathscr L\pv {#1}}}

\begin{document}

\title[Bases for pseudovarieties closed under bideterministic product]{Bases for
  pseudovarieties closed under bideterministic product}


\thanks{This work was partially supported by the Centre for Mathematics of the University of Coimbra -- UID/MAT/00324/2019, funded by the Portuguese Government through FCT/MEC and co-funded by the European Regional Development Fund through the Partnership Agreement PT2020.}

\author[A.~Costa]{Alfredo Costa}

\address{CMUC, Department of Mathematics, University of Coimbra,
  Apartado 3008, EC Santa Cruz,
  3001-501 Coimbra, Portugal.}
\email{amgc@mat.uc.pt}

\author[A.~Escada]{Ana Escada}

\address{CMUC, Department of Mathematics, University of Coimbra,
  Apartado 3008, EC Santa Cruz,
  3001-501 Coimbra, Portugal.}
\email{apce@mat.uc.pt}




\begin{abstract}  
  We show that if $\pv V$
  is a semigroup pseudovariety
  containing the finite semilattices
  and contained in $\pv {DS}$,
  then it
  has a basis
  of pseudoidentities
  between finite products of regular
  pseudowords if, and only if,
  the corresponding variety of languages
  is closed under bideterministic product.
  The key to this equivalence
  is a weak generalization of the existence and uniqueness of
  $\pv J$-reduced factorizations.
  This equational approach is used
  to address the locality of some pseudovarieties.
  In particular, it is shown that
  $\pv {DH}\cap\pv {ECom}$ is local,
  for any group pseudovariety~$\pv H$.
\end{abstract}

\maketitle


\section{Introduction}

Reiterman's theorem~\cite{Reiterman:1982}
affirms that
the pseudovarieties of semigroups are precisely
the classes of finite semigroups defined by a basis of pseudoidentities
between pseudowords.
In this paper we refine this
by showing that
the basis may be chosen to consist solely
of pseudoidentities between finite products
of regular pseudowords, whenever
$\pv V$ is a pseudovariety in the interval $[\pv{Sl},\pv {DS}]$
that is closed under bideterministic product;
motivated by this result, we call a finite product of regular pseudowords
a \emph{multiregular} pseudoword.
Conversely, we give  a proof
that every  pseudovariety of semigroups
that has a basis
of pseudoidentities
between multiregular pseudowords
is closed under bideterministic product; one may say that
this converse is already hidden in the paper~\cite{Pin&Therien:1993}, where
pseudovarieties closed under bideterministic product were first introduced,
but note that neither in~\cite{Pin&Therien:1993} nor
in the
sequels~\cite{Branco:1994,Jiang:2000,Branco:2006,Branco:2007} the profinite approach is explicitly present.
In view of these results, one
may argue that closure under bideterministic product is a relatively mild condition to impose upon a pseudovariety.
Another reason for the interest
in the bideterministic closure is that it is a natural companion
of the closure under left deterministic product
(which algebraically translates to the equality between
$\pv V$ and the Mal'cev product $\pv K\malcev\pv V$),
the closure
under right deterministic product
(that translates to $\pv V=\pv D\malcev\pv V$),
and of the closure under unambiguous product
(translated to $\pv V=\Lo I\malcev\pv D$).


Theorem~\ref{t:cut-nice-factorizations},
our main result, is the key to our refinement of Reiterman's theorem
and other results. It is
a sort of weak generalization of
the theorem on the uniqueness of
$\pv J$-reduced factorizations~\cite{Almeida&Azevedo:1993}, recalled
as Theorem~\ref{t:uniqueness-J-factorizations}.
It is also a theorem in the spirit of the solutions of
the ``pseudoword problem'' (of knowing when a pseudoidentity
is satisfied by a pseudovariety)
obtained in~\cite{Trotter&Weil:1997} for pseudovarieties closed under left, right or unambiguous product.

Some inspiration was taken
from the fact, shown
in~\cite{Almeida&ACosta:2017},
that the free profinite semigroups over $\pv V$
are equidivisible when $\pv V$
is closed under unambiguous product.
Here, we show that a weak form of equidivisibility
still stands when we only know that
$\pv V$ is closed under bideterministic
product (Theorem~\ref{t:bidet-equidivisibility}).
This is crucial for the proof of our main result.
This form of weak divisibility
is based on the notion of~\emph{good factorization},
which defines the Pin-Th\'erien
expansion, first introduced in~\cite{Pin&Therien:1993}.

The property of a pseudovariety of semigroups being local
is relevant but often difficult to prove.
In~\cite{Costa&Escada:2013} it is shown that
if $\pv V$ is a local monoidal pseudovariety
of semigroups containing $\pv {Sl}$,
then $\pv {K}\malcev\pv {V}$
and $\pv {D}\malcev\pv V$ are also local monoidal pseudovarieties
of semigroups.
Consider now the operator
$\pv V\mapsto\overline{\pv V}$
that associates to each pseudovariety
of semigroups $\pv V$
the least pseudovariety of semigroups
$\overline{\pv V}$
containing $\pv V$
that is closed under bideterministic product.
The methods used in~\cite{Costa&Escada:2013} do not carry on to this operator (see the discussion in
Section~\ref{sec:interpl-with-semid}).
But, restricting our attention
to the class $\pv {RS}$ of finite
semigroups whose set of regular elements
is a subsemigroup,
then, with our key result (Theorem~\ref{t:cut-nice-factorizations}) we do prove that
if $\pv V$ is a local monoidal pseudovariety of semigroups contained in
the interval $[\pv{Sl},\pv{DS}\cap\pv {RS}]$,
then $\overline{\pv V}$ is also a local monoidal pseudovariety of semigroups.
This implies, for example, that $\pv {DH}\cap \pv{ECom}$
is local, for every pseudovariety $\pv H$ of groups
(this family of peudovarieties has received
some attention \cite{Almeida&Weil:1994,Almeida&Weil:1994c,Auinger:2002}).

The paper is organized as follows.
After the introduction
and a section of preliminaries,
we recall in Section~\ref{sec:pin-ther-expans} the Pin-Th\'erien expansion of
a monoid, also giving
its semigroup counterpart.
The latter is because we want to work with semigroup pseudovarieties
that are non-monoidal,
such as those of the form
$\pv V\ast\pv D$, seen in Section~\ref{sec:interpl-with-semid}.
Sections~\ref{sec:good-fact-pseud} to~\ref{sec:break-fact-assum} constitute
the paper's core,
where several aspects of the notion
of good factorization of a pseudoword are explored, culminating in
the main results.
Finally,
Sections~\ref{sec:fact-glob} to~\ref{sec:interpl-with-semid} are motivated by the investigation
on the locality of pseudovarieties.


\section{Preliminaries}
\label{sec:preliminaries}


For more details on (profinite) semigroups the reader is referred
to the introductory text~\cite{Almeida:2003cshort}
and the books~\cite{Almeida:1994a,Rhodes&Steinberg:2009qt}.
The definitions and results related with monoids are similar.


For a semigroup $S$,
  the monoid  $S^I=S\uplus \{1\}$
  is obtained from $S$ by adjoining
  to $S$ a neutral element $1$ not in $S$.
  Every semigroup homomorphism $\varphi \colon S\to T$ admits an extension
  to a monoid homomorphism $\varphi^{I}\colon S^I\to T^I$ such that $\varphi ^I(1)=1$. The object $S^I$ may be different
  from the frequently used~$S^1$: the latter equals
  $S$ if $S$ is a monoid, and is $S^I$ if $S$ is not a~monoid.


  We use the standard notations for Green's equivalence relations $\R$, $\L$
  and~$\J$ and its associated quasi-orders
  $\leq_\R$, $\leq_\L$ and $\leq_\J$
  on a semigroup $S$: for $s,t\in S$,
  $s\leq _\R t$ if $s\in tS^I$, $s\leq _\L t$ if $s\in S^It$, an $s\leq _\J t$ if $s\in S^ItS^I$, and for $\K\in \{\R, \L,\J\}$,
  we have $s\mathrel{\K }t$ when $s\leq _\K t$ and $t\leq _\K s$. The elements
  $s\in S$ such that  $s\in sSs$ are said to be \emph{regular}.
  In general, the set of regular elements of a semigroup is not a subsemigroup. If $S$ is a compact semigroup (i.e., a semigroup endowed with a compact topology for which the semigroup operation is continuous),
  then, for each $\K\in \{\R, \L,\J\}$,
  a $\K$-class $K$ of $S$
  contains a regular element if and only if
  all its elements   are regular,
  in which case we say that $K$ is \emph{regular}. Moreover, 
  one also views $S^I$ as a compact semigroup
  by adding $I$ has an isolated point.

  \subsection{Pseudovarieties} \label{subsec-preamble-pseudovarieties}

  A \emph{pseudovariety of semigroups} is a class of finite semigroups closed under taking subsemigroups, homomorphic images and finitary direct products. The pseudovariety of all finite semigroups is denoted by~$\pv S$.
 We list some other pseudovarieties which have a role in this paper: $\pv {Sl}$, the pseudovariety of all finite semillatices;  $\pv  J$, the pseudovariety of all finite semigroups whose regular $\mathcal J$-classes are trivial; $\pv D$, the pseudovariety of all finite semigroups whose idempotents are right zeros; and  $\pv K$, the dual of $\pv D$.
 For any pseudovariety $\pv V$, one denotes
 by $\Lo{V}$ the pseudovariety of all finite semigroups $S$ such that $eSe\in \pv V$, for all idempotents $e\in S$,
 and $\pv{DV}$  denotes the pseudovariety of all finite semigroups  $S$ whose regular $\mathcal J$-classes are semigroups that belong to~$\pv V$.
 As we have mentioned in the introduction,
 the pseudovariety $\pv {DS}$ will play a special role in this paper:
 quite frequently, in the study of pseudovarieties,
 one has to consider the cases $\pv V\subseteq \pv {DS}$
 and $\pv V\nsubseteq \pv {DS}$
 separately.

 Let $\pv W$ and $\pv V $ be pseudovarieties of semigroups.
 The \emph{Mal'cev product} $\pv W\malcev\pv V$ is the pseudovariety of semigroups  generated by finite semigroups  $S$ for which there is a semigroup homomorphism $\varphi\colon S\to T$, for some $T\in \pv V$, such that $\varphi ^{-1}(e)\in \pv W$, for each idempotent $e \in T$.
The \emph{semidirect product} $\pv V\ast \pv W$ is the pseudovariety generated by all semidirect products of the form $S\ast T$ with $S\in \pv V$ and $T\in \pv W$. 

\subsection{Free pro-$\pv V$ semigroups}\label{subsec-preamble-pro_V-semigroups}


In what follows, finite semigroups are viewed as compact semigroups,
endowed with the discrete topology.
A compact semigroup $S$ is said to be \emph{$A$-generated}
if there is a map $\iota\colon A\to S$ such that
  $\iota (A)$ generates a dense subsemigroup of $S$.
It is said to be~\emph{residually in~$\pv V$}, where $\pv V$
  is a pseudovariety, if
  for every two distinct elements~$s_1$ and~$s_2$ of~$S$, there is some continuous homomorphism $\varphi\colon S\to T$ into a semigroup $T\in \pv V$ such that $\varphi (s_1)\neq \varphi (s_2)$.
  By a \emph{pro-$\pv V$ semigroup} we mean a compact semigroup residually in $\pv V$.

In this paper all alphabets are finite.
  For each alphabet~$A$, there is a unique (up to isomorphism)
$A$-generated free pro-$\pv V$ semigroup,  denoted~$\Om AV$, endowed with a mapping $\iota _{\pv V}\colon A\to \Om AV$,
which satisfies the following universal property:
for any map $\varphi\colon A \to S$ into a pro-$\pv V$ semigroup~$S$ there is a unique continuous homomorphism
$\widehat{\varphi}^{\pv V}\colon \Om AV\to S$ such that
$\widehat{\varphi}^{\pv V}\circ \iota  _{\pv V}=\varphi $.
The finiteness of $A$ guarantees that $\Om AV$ is metrizable.
The elements of $\Om AV$ are called \emph{pseudowords} (with respect to $\pv V$).
In particular, the map
  \mbox{$\iota _{\pv V}\colon A\to \Om A V$} induces a unique continuous homomorphism
  $p_{\pv V}\colon \Om AS\to \Om AV$,
  which is called the \emph{natural projection of $\Om AS$ onto $\Om AV$},
  and satisfies $p_{\pv V}\circ \iota _{\pv S}=\iota_{\pv V}$.
  We write $[u]_{\pv V}$ instead of  $p_{\pv V}(u)$.
  In case $\pv V=\pv S$, we
use the notation $\widehat\varphi$ for $\widehat\varphi^{\pv S}$,
and we speak of \emph{profinite semigroups} instead of pro-$\pv S$
semigroups. 
  
  Let $\om AV$ be the subsemigroup of $\Om AV$ generated by
  $\iota  _{\pv V}(A)$.
  If $\pv V$ is not the trivial pseudovariety,
  then $\iota_{\pv V}$ is injective, and so
  $A$ may actually be seen as a subset of $\Om AV$. Moreover, if $\pv V$ contains the pseudovariety $\pv N$ of finite nilpotent semigroups,
  then $\om AV$ is naturally
  isomorphic to the free semigroup~$A^+$, endowed with the discrete topology, and the elements of $\om AV$ are isolated in $\Om AV$. From hereon, we
  identify $\om AV$ with $A^+$ when $\pv V$ contains~$\pv N$.




  Let $\pv V$ be a pseudovariety
  that contains $\pv{Sl}$ and let
  $c\colon \Om A V \to \Om A{Sl}$  be the
  \emph{content mapping}, the unique continuous
  homomorphism from $\Om AV$ onto $\Om A{Sl}$
  such that $c\circ\iota_{\pv V}=\iota_{\pv{Sl}}$.
  If $u$ is a word on $A$, then
  $c(u)$ is the set of letters occurring in $u$.
  In general, $c(u)$ is the \emph{content of $u$}, for every $u\in  \Om A V$.

  A \emph{pseudoidentity} $(u=v)$ in variables of the alphabet $A$
  is a formal equality between elements $u$ and $v$ of $\Om A S$. For a profinite semigroup $S$, we  write $S\models u=v$ when  $S$ satisfies the pseudoidentity $u=v$ (that is, for every map
  $\varphi\colon A\to S$, we have $\widehat{\varphi} (u)=\widehat{\varphi}(v)$), and write $\pv V\models u=v$ when all semigroups of $\pv V$
  satisfy $u=v$.
  More generally, we write $\pv V\models \Sigma$
  when $\Sigma$ is a set of pseudoidentities satisfied by all semigroups of~$\pv V$.
  The class of finite semigroups satisfying all elements
  of $\Sigma$ is denoted~$\op\Sigma\cl$.
  Reiterman's Theorem~\cite{Reiterman:1982}
  states that the pseudovarieties of semigroups
  are precisely the classes of the form $\op \Sigma \cl$.
  As a relevant example, we have $\pv{DS}=\op ((xy)^\omega (yx)^\omega (xy)^\omega )^\omega =(xy)^\omega \cl $,
  where, if $s$ is an element
  of a profinite semigroup, $s^\omega$ is the idempotent $s^\omega=\lim s^{n!}$
  (more generally, we use the notation $s^{\omega+k}=\lim s^{n!+k}$).
  One says that $\Sigma$ is a \emph{basis} for $\pv V$
  when $\pv V=\op\Sigma\cl$.
  

  A language $L\subseteq A^+$ is \emph{$\pv V$-recognizable} if there is a homomorphism $\varphi$ from $A^+$ into a semigroup
  $S$ of $\pv V$ such that $L=\varphi ^{-1}\varphi L$.
  The following proposition establishes a link between $\pv V$-recognizable languages and the topology of $\Om AV$, when $\pv V$ contains $\pv N$.
  The restriction $\pv V\supseteq \pv N$ may be dropped, but the
  statement becomes less direct, and it suffices for us that
  $\pv V\supseteq \pv N$.

  \begin{Thm}[{cf.~\cite[Theorem 3.6.1]{Almeida:1994a}}]\label{t:Vrecognizable-clopensets}
    Let $\pv V$ be a semigroup pseudovariety containing $\pv N$.
  A language $L\subseteq A^+$
  is $\pv V$-recognizable if and only if its closure
  $\overline{L}$ in $\Om AV$ is open.
\end{Thm}

Note that, in Theorem~\ref{t:Vrecognizable-clopensets},
one has $\overline{L}\cap A^+=L$, because the elements of $A^+$
are isolated in $\Om AV$. Hence, Theorem~\ref{t:Vrecognizable-clopensets}
states that the $\pv V$-recognizable languages of $A^+$
are precisely the traces in $A^+$ of the clopen subsets
of $\Om AV$. We shall use Theorem~\ref{t:Vrecognizable-clopensets}
abundantly, without reference. 

Theorem~\ref{t:Vrecognizable-clopensets} is relevant in
the framework of Eilenberg's Theorem~\cite{Eilenberg:1976}
on the correspondence between
a semigroup pseudovariety $\pv V$ and the variety $\mathcal V$ of $+$-languages
that are $\pv V$-recognizable (recall that a $+$-language is a subset of a free semigroup, while a $*$-language is a subset of a free monoid).
For the sake of conciseness,
we  write $L\in\mathcal V$ whenever
$L$ is a $\pv V$-recognizable language of $A^+$,
instead of the more precise $L\in A^+\mathcal V$.

We shall frequently switch from the viewpoint of $+$-languages
and semigroup pseudovarieties to that of $*$-languages
and monoid pseudovarieties.

\subsection{Marked products}
\label{sec:bidet-prod}


Our departing point is the following definition,
whose first three items are nowadays classical~\cite{Pin:1986;bk}.

\begin{Defi}\label{defi:special-marked-products}
  Let $L$ and $K$ be languages of $A^*$, and let $a\in A$.
  The language $LaK$, viewed as product of the languages $L$, $\{a\}$
  and $K$ -- usually referred to as a \emph{marked product} of $L$ and $K$ --
  is said to be:
  \begin{enumerate}
  \item an~\emph{unambiguous product} when
    every element $u$ of $LaK$ has a unique factorization
    $u=u_1a u_2$ such that $u_1\in L$ and $u_2\in K$;
  \item a~\emph{left deterministic product} when every word of $LaK$
    has a unique prefix in $La$;
  \item a~\emph{right deterministic product}
    when every word of $LaK$ has a unique suffix in $aK$;
  \item a~\emph{bideterministic product}
    if the marked product $LaK$ is simultaneously right and left deterministic.
  \end{enumerate}
\end{Defi}

We say that a language $L$ of $A^*$ is a \emph{prefix code} (resp.~\emph{suffix code}) if $\forall u\in L$, $uA^+\cap L=\emptyset$ (resp.~$\forall u\in L$, 
$A^+u\cap L=\emptyset$).

Note that $LaK$ is left deterministic if and only if $La$ is a prefix code. Dually, it is
right deterministic if and only if $aK$ is a suffix code.

Next, we introduce a varietal companion of Definition~\ref{defi:special-marked-products}.


\begin{Defi}\label{defi:V-mon-closed}
  Let $\pv V$ be a pseudovariety
of monoids, and let $\mathcal V$ be the correspondent variety
of $\pv V$-recognizable $*$-languages.
Then $\pv V$ is said to be:
\begin{enumerate}
\item \emph{closed under unanbiguous product}
  if $LaK\in\mathcal V$ whenever $L,K\in \mathcal V$, $a$ is a letter and $LaK$ is an unambiguous product;
\item \emph{closed under
  left deterministic product} if $LaK\in \mathcal V$
  whenever~\mbox{$L,K\in \mathcal V$}, $a$ is a letter and $La$ is a prefix code;
\item \emph{closed under right deterministic product} if $LaK\in \mathcal V$
  whenever~\mbox{$L,K\in \mathcal V$}, $a$ is a letter and $aK$ is a suffix code;
\item \emph{closed under bideterministic product} if~\mbox{$LaK\in \mathcal V$}
  whenever $L,K\in \mathcal V$, $a$~is a letter,  $La$ is a prefix code and $aK$ is a suffix code.
\end{enumerate}
\end{Defi}

It is well known that we have
the following characterization of the first three types
of pseudovarieties mentioned in Definition~\ref{defi:V-mon-closed}.

\begin{Thm}[\cite{Pin:1980b,Pin:1986;bk,Pin&Straubing&Therien:1988,Straubing:1979a},
   see also survey~\cite{Pin:1997}]\label{t:characterization-of-left-right-deterministic-product.}
  Let $\pv V$ be a pseudovariety of monoids.
  Then:
  \begin{enumerate}
  \item $\pv V$ is closed under unambiguous product
  if and only if $\pv V=\Lo I\malcev\pv V$;
  \item $\pv V$ is closed under left deterministic product
    if and only if\/ $\pv V=\pv K\malcev\pv V$;
  \item $\pv V$ is closed under right deterministic product
  if and only if\/~\mbox{$\pv V=\pv D\malcev\pv V$}.
  \end{enumerate}
\end{Thm}



On the other hand, pseudovarieties closed
under bidetermistic product
have no characterization
that, like in Theorem~\ref{t:characterization-of-left-right-deterministic-product.}, uses a Mal'cev product.

In this section, we have so far stayed in
the realm of $*$-languages, and in the corresponding
one of monoid pseudovarieties. But the definitions
and results  we reviewed have natural companions in the realms of $+$-varieties
and semigroup pseudovarieties.
A way to define these counterparts is
by restricting each language $L$ and $K$ to be either
a $+$-language or the language $\{1\}$.
We next see a concrete manifestation of this in the case of the bidetermistic product, the subject of attention in this paper.

\begin{Defi}\label{defi:V-sgp-bidet-closed}
  Let $\pv V$ be a semigroup pseudovariety,
  and $\mathcal V$ be the variety
  of $\pv V$-recognizable $+$-languages.
Then $\pv V$
is \emph{closed under bideterministic product}
if $LaK\in \mathcal V$ when $a$ is a letter and the next conditions are satisfied:
$L\in \mathcal V$ or $L=\{1\}$; $K\in \mathcal V$ or $K=\{1\}$; $La$ is a prefix code;
$aK$ is a suffix~code.
\end{Defi}

For $u\in A^*$, let $\be 1(u)$ be the first letter  of $u$ if $u\neq 1$, and $\be 1(1)=1$. Dually, $\te 1(u)$ is the last letter of $u$ if $u\neq 1$, and $\te 1(1)=1$. The maps $\be 1$ and $\te 1$
extend uniquely to continuous maps from $(\Om AS)^1$ into $A\cup\{1\}$.

\begin{Remark}\label{r:products-of-several-codes}
  The product of prefix codes is a prefix code.
  In particular, $Lu$ is a prefix code when
  $L$ is a prefix code, whenever $u\in A^*$.
  Dual remarks hold for suffix codes.
  This implies that,
  still assuming
  that $\pv V$ is closed under bideterministic product
  and that $\mathcal V$ is its corresponding variety of $+$-languages,
  if $L_1\,L_2,\ldots,L_n\in\mathcal V$,
  and $u_0,u_1,\ldots,u_n\in A^*$, with $u_i\neq 1$ when $i\neq 0$ and $i\neq n$,
  are such that
  $L_i\cdot \be 1(u_i)$ is a prefix code
  and $\te 1(u_{i-1})\cdot L_i$ is a suffix
  code for every $i\in \{1,\ldots,n\}$, 
  then $u_0L_1u_1L_2u_2\cdots u_{n-1}L_nu_n\in \mathcal V$.
\end{Remark}

\section{The Pin-Th\'erien expansion}
\label{sec:pin-ther-expans}

\subsection{The Pin-Th\'erien expansion of a finite monoid}
\label{sec:monoid-version}

In what follows, we consider
an onto homomorphism of monoids \mbox{$\varphi\colon A^*\to M$} such that $M$
is finite. The finiteness of $M$ is
frequently not important (sometimes it is, like in Theorem~\ref{t:injective-restriction-of-the-canonical-projection}), but that is the framework in which we are interested, and it is a general assumption also made in~\cite{Pin&Therien:1993}.

\begin{Defi}
  A \emph{good factorization} (with respect to $\varphi$) is
  a triple $(x_0,a,x_1)$ of $A^*\times A\times A^*$
such that $\varphi(x_0a)<_{\R}\varphi(x_0)$
and \mbox{$\varphi(ax_1)<_{\L}\varphi(x_1)$}.
  Two good factorizations $(x_0,a,x_1)$ and $(y_0,b,y_1)$
are said to be \emph{equivalent} when $\varphi(x_0)=\varphi(y_0)$, $a=b$
and $\varphi(x_1)=\varphi(y_1)$. A \emph{good factorization of $x\in A^*$}
is a good factorization $(x_0,a,x_1)$
such that $x=x_0\cdot a\cdot x_1$.
\end{Defi}

It is shown in~\cite{Branco:2006} that, for all $x,y\in A^*$, every good factorization of $x$ is equivalent to at most one good factorization of $y$.

\begin{Defi}\label{def:congruence}
  Let $\sim_{\varphi}$ be the relation on $A^*$ defined by $x\sim_{\varphi} y$
if and only if the following conditions are satisfied:
\begin{enumerate}
\item $\varphi(x)=\varphi(y)$;
\item each good factorization of $x$ is equivalent to a good factorization
  of~$y$;
\item each good factorization of $y$ is equivalent to a good factorization
  of~$x$.
\end{enumerate}
\end{Defi}

The relation~$\sim_{\varphi}$ is a congruence~\cite{Pin&Therien:1993}.

\begin{Defi}
Denote $A^*/{\sim_{\varphi}}$ by $M_\varphi$,
and by $\varphi_{\pv {bd}}$ the corresponding
quotient morphism from $A^*$ onto $A^*/{\sim_{\varphi}}$.
We say that $M_\varphi$ is the \emph{Pin-Th\'erien expansion of $M$
with respect to~$\varphi$}.
\end{Defi}

Note that there is a unique
onto monoid homomorphism
$p_{\varphi}\colon M_\varphi\to M$
such that $\varphi=p_\varphi\circ\varphi_{\pv{bd}}$.
Note also that the finiteness of $M$
guarantees the finiteness of $M_\varphi$~\cite{Pin&Therien:1993}.
The correspondence $\varphi\mapsto\varphi_{\pv {bd}}$
is indeed an expansion cut to generators in the sense of Birget and Rhodes,
as shown in~\cite{Pin&Therien:1993}.
In fact, it is proved in~\cite{Branco:2006} that it is an expansion
in a broader sense.





\begin{Defi}\label{def:Vbd}
  For a monoid pseudovariety $\pv V$,
  denote by $\pv V_{\pv {bd}}$ the
monoid pseudovariety generated by Pin-Th\'erien expansions of monoids in~$\pv V$.
We say that $\pv V$ is \emph{closed under Pin-Th\'erien expansion}
when~$\pv V=\pv V_{\pv {bd}}$.
\end{Defi}

\begin{Thm}[{\cite[Corollary~4.5]{Pin&Therien:1993}}]\label{t:V-Vbd-monoid-perspective}
  Let\/ $\pv V$ be a monoid pseudovariety. Then
  $\pv V=\pv V_{\pv {bd}}$
  if and only if $\pv V$ is closed under bideterministic product.
\end{Thm}

The intersection of a family of monoid pseudovarieties closed under bideterministic
product is a pseudovariety closed under bideterministic product. Hence,
for each monoid pseudovariety $\pv V$ we may consider
the least monoid pseudovariety $\overline{\pv V}$
closed under bideterministic product and containing~$\pv V$.

\begin{Remark}\label{r:tower-of-bds}
  Consider the chain
    of pseudovarieties $(\pv V_n)_{n\geq 0}$
    recursively defined by
    $\pv V_0=\pv V$ and $\pv V_n=(\pv V_{n-1})_{\pv {bd}}$ for each $n\geq 1$.
    Then $\bigcup_{n\geq 0}\pv V_n$
    is closed for the bideterministic product
    (cf.~Theorem~\ref{t:V-Vbd-monoid-perspective})
    and $\overline{\pv V}=\bigcup_{n\geq 0}\pv V_n$.
  \end{Remark}

  \begin{Example}\label{eg:expansion-of-a-group}
    It is easy to see that
    the Pin-Th\'erien expansion of a
    finite group~$G$, viewed as a monoid, is $G$ itself.
    Hence, viewing a pseudovariety of groups~$\pv H$ as a monoid pseudovariety,
    one has $\pv H=\pv H_{\pv {bd}}=\overline{\pv H}$.
  \end{Example}
  
\begin{Example}\label{eg:closure-of-Sl-and-others}
  Let $\pv {Ecom}$ be the pseudovariety of monoids whose idempotents commute.  In~\cite{Pin&Therien:1993} it is shown that
  $\overline{\pv {Sl}\vee \pv H}=\pv {DH}\cap\pv {ECom}$,
  whenever $\pv H$ is a pseudovariety
  of groups. In particular, the equality
  $\overline{\pv {Sl}}=\pv J\cap \pv {ECom}$ holds.
\end{Example}

For a pseudovariety $\pv V$ of monoids, we denote by $\pv{RV}$ the class of finite monoids whose set of regular elements is a submonoid in $\pv V$. When $\pv V\subseteq \pv {CR}$,  $\pv {RV}$ is a pseudovariety, where  
$\pv {CR}$ is the pseudovariety of all finite
monoids which are completely regular.

\begin{Example}\label{eg:closure-of-CR-and-others}
  If $\pv V\subseteq \pv {CR}$, then
  $    \overline{\pv V}=\pv{RV}\cap \op x^\omega(xy)^\omega
        =(xy)^\omega=(xy)^\omega y^\omega\cl$.
        This formula was deduced in~\cite{Jiang:2000}
        from its special case, proved in~\cite{Branco:1994},
        in which~$\pv V\subseteq \pv {B}$.
\end{Example}

\subsection{The Pin-Th\'erien expansion of a finite semigroup}

In the following lines, we  define a semigroup version of the Pin-Th\'erien expansion.
Let $\varphi\colon A^+\to S$ be an onto homomorphism of
semigroups, with $S$ finite. Consider
the monoid homomorphism $\varphi^I\colon A^*\to S^I$ such that
$\varphi ^I(u)=\varphi(u)$, when $u\neq 1$.
Then the empty word $1$ of $A^*$
is the unique element of the $\sim_{\varphi^I}$-class of $1$.
Therefore,
if $I'=(\varphi^I)_{\pv {bd}}(1)$ is the identity of $(S^I)_{\varphi^I}$,
then the semigroup $S_\varphi=(\varphi^I)_{\pv {bd}}(A^+)$
satisfies $S_\varphi=(S^I)_{\varphi^I}\setminus\{I'\}$,
and so
\begin{equation*}
  (S_\varphi)^I=(S^I)_{\varphi^I},
\end{equation*}
where we are making the identification $I=I'$.
We may then consider the semigroup homomorphism
$\varphi_{\pv{bd}}\colon A^+\to S_\varphi$
obtained by the restriction of
$(\varphi^I)_{\pv{bd}}$ to $A^+$.
Note that
\begin{equation*}
  (\varphi_{\pv{bd}})^I=(\varphi^I)_{\pv{bd}}.
\end{equation*}
The semigroup $S_\varphi$ is the \emph{(semigroup) Pin-Th\'erien expansion
of~$S$ with respect to $\varphi$}. We then let $\sim_{\varphi}$ be the kernel of
$\varphi_{\pv{bd}}$. We also denote by $p_\varphi$
the unique onto semigroup homomorphism
$p_{\varphi}\colon S_\varphi\to S$
such that $\varphi=p_\varphi\circ\varphi_{\pv{bd}}$.
Note also that, for such a $\varphi$, we have
$(p_\varphi)^I=p_{\varphi^I}$.

\begin{Remark}
  The expansion $\varphi_{\pv{bd}}\colon A^+\to S_{\pv{bd}}$
  may equivalently be defined by adapting
  the definitions given in Subsection~\ref{sec:monoid-version},
  by letting $(x_0,a,x_1)$
  be a good factorization of $x_0ax_1\in A^+$ whenever
  $\varphi (x_0a)<_{\mathcal{R}}\varphi^I(x_0)$
  and  $\varphi (ax_1)<_{\mathcal{L}}\varphi^I(x_1)$.
\end{Remark}


For a pseudovariety of monoids $\pv V$, one
denotes by $\pv V_{\pv S}$ the least pseudovariety
of semigroups containing $\pv V$.
It is well known that $S\in\pv V_{\pv S}$ if and only if
$S^1\in\pv V$. Moreover, if $\pv V$ contains $\pv {Sl}$,
then $S\in\pv V_{\pv S}$ if and only if
$S^I\in\pv V$, a fact that we shall use
in the proof of the next proposition.

\begin{Prop}\label{p:Vbd-semigroup-monoid-version}
  If $\pv V$ is a pseudovariety of monoids
  containing $\pv {Sl}$,
  then the equality $(\pv V_{\pv S})_{\pv {bd}}=(\pv V_{\pv {bd}})_{\pv S}$
  holds.
\end{Prop}

\begin{proof}
  Let $\psi\colon A^*\to M$ be a surjective
  homomorphism onto a monoid $M$ of~$\pv V$. The
  proof of $(\pv V_{\pv S})_{\pv {bd}}\supseteq \pv V_{\pv {bd}}$,
  equivalently of
  $(\pv V_{\pv S})_{\pv {bd}}\supseteq (\pv V_{\pv {bd}})_{\pv S}$,
  is concluded once we show
  that $M_{\psi}$, viewed as a semigroup, belongs to $(\pv V_{\pv S})_{\pv {bd}}$.

  Denote by $1_M$ the identity of $M$. We consider two cases.

  Suppose first that $\psi^{-1}(1_M)=\{1\}$. Let $B=A\uplus\{b\}$, where $b$ is a new letter not in $A$.
  Consider the homomorphism $\lambda_b\colon B^*\to A^*$
  such that
  $\lambda_b(b)=1$ and $\lambda_b(a)=a$ for every $a\in A$.
  Denote $\lambda_b(u)$ by $\overline{u}$.
  Consider also
  the semigroup homomorphism
  $\psi_b\colon B^+\to M$
  such that
  $\psi_b(b)=1_M$ and $\psi_b(a)=\psi(a)$.
    We claim that, in the category of semigroups,
  $M_\psi$
  is a homomorphic
  image of $M_{\psi_b}$.
  For that purpose, we collect the following series of facts:
  \begin{enumerate}
  \item We have $b^+=\psi_b^{-1}(1_M)$ and
    $\{1\}=\psi^{-1}(1_M)$, thus $b^+\setminus \{b\}$
    and $\{b\}$
  are $\sim_{\psi_b}$-classes, 
  and $\{1\}$ is  a $\sim_\psi$-class.
 \item If $(x_0,a,x_1)\in B^*\times B\times B^*$ is a good factorization
   with respect to~$\psi_b$ of an element of $B^+\setminus b$, then $a\neq b$.
 \item Let  $(x_0,a,x_1)\in B^*\times A\times B^*$. Then $(x_0,a,x_1)$
   is a good factorization with respect to $(\psi_b)^I$ if and only if
   $(\overline{x_0},a,\overline{x_1})$ is a good
   factorization with respect to $\psi$.
   This is trus because $(\psi_b)^I(x_0)=\psi (\overline{x_0})$
   unless $x_0=1$, in which case $(\psi_b)^I(1\cdot a)<_{\R}(\psi_b)^I(1)=I$
   and $\psi(1\cdot a)<_{\R}\psi(1)=1_M$ hold,
   and because of the dual phenomena concerning
   the third component of the factorizations.
 \end{enumerate}
 Taking into account the partition
 $B^+=\{b\}\uplus (b^+\setminus\{b\})\uplus B^*AB^*$, it follows that
  the map from
  $M_{\psi_b}=B^+/{\sim_{\psi_b}}$ to $M_\psi=A^*/{\sim_{\psi}}$,
  sending $u/{\sim_{\psi_b}}$ to $\overline{u}/{\sim_{\psi}}$,
  is a well defined onto homomorphism of semigroups, thus establishing
  the~claim.
   As $M\in \pv V_{\pv S}$ and $M_{\psi_b}\in(\pv V_{\pv S})_{\pv {bd}}$,
   it follows
   that
   $M_\psi\in (\pv V_{\pv S})_{\pv {bd}}$.

  Suppose now that $\psi^{-1}(1_M)\neq\{1\}$.
  Then $\psi(A^+)=M$.
  Let $\Psi$ be the monoid
  homomorphism from $A^*$ onto $M^I$
  such that $\Psi(u)=\psi(u)$
  for every $u\in A^+$.
  Note that $\Psi$ is onto and that
  $\Psi^{-1}(I)=\{1\}$.
  Moreover, since $\pv {Sl}\subseteq \pv V$,
  the monoid $M^I$ belongs to $\pv V$.
  Hence, by the already proved case, we have $(M^I)_\Psi\in (\pv V_{\pv S})_{\pv {bd}}$.
  Let $\pi$ be the onto monoid homomorphism
  from $M^I$ to $M$ whose restriction
  to $M$ is the identity. Then $\pi\circ\Psi=\psi$.
  Because we are dealing with an expansion cut to generators,
  this implies that $M_\psi$ is a homomorphic image of $(M^I)_\Psi$,
  whence $M_\psi\in (\pv V_{\pv S})_{\pv {bd}}$.

  Finally, we show that
  $(\pv V_{\pv S})_{\pv {bd}}
  \subseteq
  (\pv V_{\pv {bd}})_{\pv S}$.
  Let $S\in \pv V_{\pv S}$,
  and let $\varphi\colon A^+\to S$ be an onto homomorphism.
  Since $\pv V$ contains~$\pv {Sl}$,
  we have $S^I\in \pv V$,
  whence $(S^I)_{\varphi^I}\in \pv V_{\pv {bd}}$.
  But $(S^I)_{\varphi^I}=(S_{\varphi})^I$,
  whence $S_\varphi\in (\pv V_{\pv {bd}})_{\pv S}$.
\end{proof}

We omit the proof
of the following
theorem,
since it can be made by just imitating
the proof in~\cite{Pin&Therien:1993} of its monoid analog (Theorem~\ref{t:V-Vbd-monoid-perspective}).

\begin{Thm}\label{t:semigroup-VVbd}
  A pseudovariety of semigroups $\pv V$
  satisfies~$\pv V=\pv V_{\pv{bd}}$ if and only
  if\/ $\pv V$ is closed under bideterministic product.
\end{Thm}

As for monoids, let $\overline{\pv V}$ be
the least semigroup pseudovariety, containing
the semigroup pseudovariety $\pv V$,
which is closed under bideterministic product.
Note that Remark~\ref{r:tower-of-bds}
also holds for pseudovarieties of
semigroups. This fact
and Proposition~\ref{p:Vbd-semigroup-monoid-version}
yeld the following corollary.

\begin{Cor}\label{c:overline-commutes}
  If $\pv V$ is a
  monoid pseudovariety containing $\pv {Sl}$,
  then the equality $\overline{\pv V_{\pv S}}
  =\overline{\pv V}_{\pv S}$
  holds.\qed
\end{Cor}

In Corollary~\ref{c:overline-commutes},
the hypothesis $\pv V\supseteq\pv {Sl}$ is
needed as seen below.

\begin{Example}\label{eg:semigroup-PT-expansion-of-I}
  Let $\pv I$ be the class of trivial semigroups.
  Viewing $\pv I$ as a pseudovariety of monoids, one has
  $\overline{\pv I}=\pv I$ (Example~\ref{eg:expansion-of-a-group}),
  but if we view $\pv I$ as a pseudovariety of semigroups,
  then we get $\overline{\pv I}=\pv N$.
  Indeed, $\overline{\pv I}\subseteq \pv N$
  because an $\pv N$-recognizable language
  is either finite or co-finite, and since
  a co-finite language can neither be a prefix code nor a suffix code,
  one clearly has that $\pv N$ is closed under bideterministic product.
  On the other hand, $\pv N\subseteq \overline{\pv I}$
  since every
  finite language is
  the finite union of bideterministic products
  of the form $\{1\}\cdot a_1\cdot \{1\}\cdot
  a_2\cdots a_{n-1}\cdot \{1\}\cdot a_n\{1\}$, with $a_1,\ldots,a_n$
  letters.
\end{Example}

\section{Good factorizations of pseudowords}
\label{sec:good-fact-pseud}

Inspired by the concept of good factorization of a word,
we define an analog for pseudowords.

\begin{Defi}\label{defi:good-factorization-pseudowords}
  Let $\pv V$ be a pseudovariety of semigroups.
  A \emph{good factorization} of an element $x$
  of $\Om AV$ is a triple $(\pi,a,\rho)$ in $(\Om AV)^1\times A\times (\Om AV)^1$,
  with $x=\pi a \rho$, such that $\pi a<_\R\pi$ and $a\rho<_\L \rho$.
\end{Defi}

We omit next lemma's easy proof,
analog to that \mbox{of~\cite[Lemma 2.1]{Pin&Therien:1993}.}

\begin{Lemma}\label{l:shortening-good-factorizations}
  Let $\pv V$ be a pseudovariety of semigroups.
  Suppose that $(\pi,a,\rho)$ is a good factorization of
  an element of $\Om AV$. Let $\pi'$ be a suffix of $\pi$ and let $\rho'$ be a prefix of $\rho$. Then $(\pi',a,\rho')$ is a good factorization.
\end{Lemma}

In the next definition the hypothesis
$\pv V\supseteq \pv N$
is required to
ensure that
 the elements of $A^+$ embed in $\Om AV$ (as isolated points).
 

\begin{Defi}\label{defi:plus-good-factorization-pseudowords}
  Let $\pv V$ be a semigroup pseudovariety containing $\pv N$.
  A \emph{$+$-good factorization} of $x\in\Om AV$
  is a triple $(\pi,u,\rho)$ in $(\Om AV)^1\times A^+\times (\Om AV)^1$, with $x=\pi u\rho$, such that $\pi\,\be 1(u)<_\R\pi$ and $\te 1(u)\,\rho<_\L \rho$.
\end{Defi}

Note that a good factorization is $+$-good factorization
(assuming \mbox{$\pv V\supseteq \pv N$}).

Starting in the next lemma, we use the usual notation $B(\pi,\varepsilon)$
for the open ball of center $\pi$ and radius $\varepsilon$.

  \begin{Lemma}\label{l:fall}
    Let $\pv V$ be a pseudovariety
    of semigroups containing~$\pv N$.
  Suppose that $\pi\in(\Om AV)^1$ and $a\in A$ are such that
  $\pi a<_\R \pi$.
  Then, there is a positive integer $k_0$
  such that, for every $k\geq k_0$,
  and for every $u\in A^\ast$,
  the set
  \begin{equation*}
    \Bigr[B\Bigl(\pi,\frac{1}{k}\Bigr)\cap A^\ast\Bigl]\cdot au
  \end{equation*}
  is a prefix code.
\end{Lemma}

\begin{proof}
  The proof reduces immediately to the case $u=1$ (cf.~first two sentences in Remark~\ref{r:products-of-several-codes}).
  Let~$J$ be the set of positive integers $k$
    for which $\Bigr[B\Bigl(\pi,\frac{1}{k}\Bigr)\cap A^\ast\Bigl]\cdot a$
    is not a prefix code.
    Suppose that $J$ is infinite.
     For each $k\in J$, we may consider
    distinct elements $w_k$ and $z_k$
    of $B\Bigl(\pi,\frac{1}{k}\Bigr)\cap A^\ast$
    such that $w_ka$ is a prefix of $z_ka$.
    For such elements, there is $t_k\in A^\ast$ with
    $z_k=w_kat_k$.
    Note that the sequences $(z_k)_{k\in J}$
    and $(w_k)_{k\in J}$
    converge to $\pi$ and that $(t_k)_{k\in J}$
    has some accumulation point $t$ in the compact space
    $(\Om AV)^1$.
    Hence, we have the equality $\pi=\pi at$,
    contradicting the hypothesis
    that $\pi a<_\R \pi$.
    To avoid the contradiction, the set $J$ must be finite.
\end{proof}

In the following proofs, we shall frequently use, without
reference, that
if~$\pv V$
is a semigroup pseudovariety
closed under bideterministic product, then~$\pv V$ contains $\pv N$ (cf.~Example~\ref{eg:semigroup-PT-expansion-of-I}).


  \begin{Cor}\label{c:fall}
  Let $\pv V$ be a pseudovariety of semigroups closed under bideterministic product.
  Suppose that $\pi\in(\Om AV)^1$ and $a\in A$ are such that
  $\pi a<_\R \pi$.
  Then, there is a positive integer $k_0$
  such that, for every $k\geq k_0$,
  and for every $u\in A^\ast$,
  the set $B\Bigl(\pi,\frac{1}{k}\Bigr)\cdot au$
  is clopen.
\end{Cor}

\begin{proof}
  By Lemma~\ref{l:fall}, there is
  a positive integer $k_0$ such that
  for every $k\geq k_0$,
    the set  $\Bigr[B\Bigl(\pi,\frac{1}{k}\Bigr)\cap A^\ast\Bigl]\cdot a$
    is a prefix code.
    Hence, for $k\geq k_0$ and $u\in A^\ast$, the language
    $\Bigr[B\Bigl(\pi,\frac{1}{k}\Bigr)\cap A^\ast\Bigl]\cdot a\cdot u$
    is a bideterministic product of
    $\pv V$-recognizable languages,
    and thus it is itself $\pv V$-recognizable
    by the hypothesis that $\pv V$ is closed under bideterministic product.
    Taking the topological closure in $\Om AV$,
    we conclude that $B\Bigl(\pi,\frac{1}{k}\Bigr)\cdot au$ is clopen.
\end{proof}

\begin{Lemma}\label{l:good-and-+-good}
  Let $\pv V$ be a pseudovariety of semigroups closed under bideterministic product.
  Suppose that $(\pi,u,\rho)$ is a $+$-good factorization
  of~$x\in\Om AV$.
  Let $a,b\in A$, $v,w\in A^\ast$
  be such that $u=av=wb$.
  Then $(\pi,a,v\rho)$
  and $(\pi w,b,\rho)$
  are good factorizations of $x$.
\end{Lemma}

\begin{proof}
  By symmetry, it suffices to show that
  $(\pi,a,v\rho)$ is a good factorization.
  We suppose that $\rho\neq 1$
  and $v\neq 1$, as
  both the cases $\rho=1$ and $v=1$ are trivial.
  We only need to show
  $av\rho<_{\L}v\rho$.
  Suppose on the contrary that
  $av\rho\mathrel{\L}v\rho$.
  Then $v\rho=z av\rho$
  for some $z\in \Om AV $.
  Let $(z_n)_n$ and $(\rho_n)_n$ be sequences of elements of $A^+$
  respectively converging to $z$ and $\rho$.
  Thanks to the dual of Corollary~\ref{c:fall}, there
  is a positive integer $k_0$
  such that, for every $k\geq k_0$,
  the set
  $v\cdot B\Bigl(\rho,\frac{1}{k}\Bigr)$
  is a clopen neighborhood of $v\rho$.
  Therefore, as
  $z_nav\rho_n$ converges to
  $zav\rho=v\rho$, we can build subsequences $(z_{n_k})_{k\geq k_0}$
  and $(\rho_{n_k})_{k\geq k_0}$ such that
  $z_{n_k}av\rho_{n_k}\in v\cdot B\Bigl(\rho,\frac{1}{k}\Bigr)$,
  with $\rho_{n_k}\in B\Bigl(\rho,\frac{1}{k}\Bigr)$,
  for every $k\geq k_0$.
  But then $z_{n_k}av\rho_{n_k}$ is an element of
  the intersection
  \begin{equation*}
    A^+\cdot v\Bigr[B\Bigl(\rho,\frac{1}{k}\Bigr)\cap A^\ast\Bigl]
    \cap v\Bigr[B\Bigl(\rho,\frac{1}{k}\Bigr)\cap A^\ast\Bigl],
  \end{equation*}
  for every $k\geq k_0$.
  This contradicts
  the dual of Lemma~\ref{l:fall}.
\end{proof}

  \begin{Prop}\label{p:a-sort-of-open-multiplication-+}
    Let $\pv V$ be a pseudovariety of semigroups closed under bideterministic
    product. Suppose that $(\pi,u,\rho)$ is a $+$-good factorization of
    an element $x$ of\/ $\Om AV$.
    For each positive integer $k$,
    consider the subset $L_k(\pi,u,\rho)$ of
    $\Om AV$
    defined by:
    \begin{equation*}
      L_k(\pi,u,\rho)=B\Bigl(\pi,\frac{1}{k}\Bigr)\cdot u\cdot
      B\Bigl(\rho,\frac{1}{k}\Bigr).
    \end{equation*}
    For all sufficiently large $k$,
    the set $L_k(\pi,u,\rho)$
    is a clopen neighborhood of $x$.
  \end{Prop}

  \begin{proof}
        Since $\Om AV\setminus A^+$ is an ideal of $\Om AV$,
    we have
    \begin{equation}\label{eq:a-sort-of-open-multiplication-1}
      L_k(\pi,u,\rho)\cap A^+
      =
      \Bigr[B\Bigl(\pi,\frac{1}{k}\Bigr)\cap A^\ast\Bigl]
      \cdot u
      \cdot
      \Bigr[B\Bigl(\rho,\frac{1}{k}\Bigr)\cap A^\ast\Bigl].
    \end{equation}
    By Lemma~\ref{l:fall},
    there is a positive integer $k_1$
    such that, for every $k\geq k_1$,
    the set  $\Bigr[B\Bigl(\pi,\frac{1}{k}\Bigr)\cap A^\ast\Bigl]\cdot \be 1(u)$
    is a prefix code.
    By the dual of Lemma~\ref{l:fall}, there is a positive integer $k_2$,
    such that, for every $k\geq k_1$,
    the set
    $\te 1(u)\cdot \Bigr[B\Bigl(\pi,\frac{1}{k}\Bigr)\cap A^\ast\Bigl]$
    is a suffix code.
    Therefore, for $k\geq \max\{k_1,k_2\}$,
    the product in the right side of~\eqref{eq:a-sort-of-open-multiplication-1}
    is $\pv V$-recognizable (cf.~Remark~\ref{r:products-of-several-codes}), and
    so its closure in $\Om AV$, the set $L_k(u,a,v)$,
    is open.
  \end{proof}

  The proof of the following proposition is an adaptation of part
  of the proof of~\cite[Lemma 3.2]{Almeida&ACosta&Costa&Zeitoun:2017},
  where a description of compact metric semigroups with
  open multiplication is given in terms of a property of sequences.

 \begin{Prop}\label{p:consequences-on-sequences}
    Let $\pv V$ be a semigroup pseudovariety closed under bideterministic
    product. Suppose that $(\pi,u,\rho)$ is a $+$-good factorization of
    $x\in \Om AV$.
    Let $(x_n)_n$ be a sequence of elements of\/ $\Om AV$ converging to $x$.
    There are sequences $(\pi_n)_n$
    and $(\rho_n)_n$ in $(\Om AV)^1$,
    respectively converging to $\pi$ and $\rho$,
    such that $x_n=\pi_n u\rho_n$
    for all sufficiently large $n$.
  \end{Prop}

  \begin{proof}
    For each integer $k\geq 1$,
    consider the set $L_k(\pi,u,\rho)$
    as in Proposition~\ref{p:a-sort-of-open-multiplication-+},
    and let $k_0$ be such that $L_k(\pi,u,\rho)$ is a clopen neighborhood
    of $x$ for every $k\geq k_0$ (such $k_0$ exists by Proposition~\ref{p:a-sort-of-open-multiplication-+}).
    For each $k\geq k_0$, take $p_k\in\ZZ^+$ such that
    $x_n\in L_k(\pi,u,\rho)$ when $n\geq p_k$.
    Let $(n_k)_{k\geq k_0}$ be the strictly increasing sequence
    defined by $n_{k_0}=p_{k_0}$ and $n_k=\max\{n_{k-1}+1,p_k\}$
    for $k>k_0$.
    When $n_k\leq n<n_{k+1}$,
    take $\pi_n\in B\Bigl(\pi,\frac{1}{k}\Bigr)$
    and $\rho_n\in B\Bigl(\rho,\frac{1}{k}\Bigr)$
    such that $x_n=\pi_nu\rho_n$,
    which we can do as $x_n\in L_k(\pi,u,\rho)$.
    If $n<n_{k_0}$, take $\pi_n=\rho_n=1$.
    Clearly, $(\pi_n)_n$ and $(\rho_n)_n$ respectively
    converge to $\pi$ and $\rho$.
  \end{proof}


  We are ready to prove the next theorem,
  a sort of generalization of the equidivisibility property, observed in~\cite{Almeida&ACosta:2017}, of
  the finitely generated
  free profinite semigroups
  over pseudovarieties closed under unambiguous product.

    \begin{Thm}\label{t:bidet-equidivisibility}
    Let $\pv V$ be a pseudovariety of semigroups closed under bideterministic
    product. Suppose that
    $(u,a,v)$ is a good factorization of
    an element~$x$ of~$\Om AV$.
    Let $x=wbz$ be a factorization of $x$ such that $b\in A$.
    Then, at least one of the three following cases occurs:
    \begin{enumerate}
    \item $u=w$, $a=b$ and $v=z$;\label{item:bidet-equidivisibility-1}
    \item $u=wbt$ and $z=tav$ for some $t\in(\Om AV)^1$;\label{item:bidet-equidivisibility-2}
    \item $w=uat$ and $v=tbz$ for some $t\in(\Om AV)^1$.\label{item:bidet-equidivisibility-3}
    \end{enumerate}
  \end{Thm}

  \begin{proof}
     As $A^*$ is dense in $(\Om AV)^1$,
     we may consider sequences $(w_n)_n$
     and $(z_n)_n$ of elements of $A^*$
     respectively converging to $w$ and $z$.
     Let $x_n=w_nbz_n$. Note that $\lim x_n=x$,
     and so, by Proposition~\ref{p:consequences-on-sequences},
     there are sequences $(u_n)_n$
     and $(v_n)_n$ of elements of $A^*$,
     respectively converging to $u$
     and $v$,  and a positive integer~$p$, such that $u_nav_n=w_nbz_n$, $n\geq p$.
     Since the latter is an equality of words of $A^*$,
     for each $n\geq p$ one of the following three situations occurs:
    \begin{enumerate}
      [label=(\alph*),series=axioms]
    \item $u_{n}=w_n$, $a=b$ and $v_{n}=z_n$;\label{item:bidet-equidivisibility-1ap}
    \item $u_{n}=w_nbt_{n}$ and $z_n=t_{n}av_{n}$ for some
      $t_{n}\in A^\ast$;\label{item:bidet-equidivisibility-2ap}
    \item $w_n=u_{n}at_{n}$ and $v_{n}=t_{n}bz_n$
      for some $t_{n}\in A^\ast$.\label{item:bidet-equidivisibility-3ap}
    \end{enumerate}
    We let $J_1$, $J_2$ and $J_3$ be the sets of positive integers $n$
    greater or equal than~$p$
    for which, respectively, situations~\ref{item:bidet-equidivisibility-1ap},
    \ref{item:bidet-equidivisibility-2ap} and
    \ref{item:bidet-equidivisibility-3ap} occur.
    At least one of the three sets is infinite.
    Suppose that $J_2$ is infinite.
    For each $n\in J_2$, let $t_{n}\in A^\ast$
    be as in \ref{item:bidet-equidivisibility-2ap}.
    By compactness, the sequence $(t_{n})_{n\in J_2}$
    has some accumulation point $t$ in
    $(\Om AV)^1$.
    Taking limits, we get $u=wbt$ and $z=tav$, and so
    if $J_2$ is infinite then Case~\eqref{item:bidet-equidivisibility-2}
    holds.
    Arguing in a similar manner, we conclude that
    Case~\eqref{item:bidet-equidivisibility-3} holds
    if  $J_3$ is infinite, and that
    Case~\eqref{item:bidet-equidivisibility-1} holds
    if $J_1$ is infinite.
  \end{proof}





\section{Pseudowords without good factorizations}

An analog of the next proposition, and of the corollary
following it, is implicitly proved in~\cite{Pin&Therien:1993}
for good factorizations with respect to a homomorphism defined
in a free monoid (cf.~proof of \cite[Theorem~2.6]{Pin&Therien:1993}).

  \begin{Prop}\label{p:product-of-not-good-factorization}
  Let $\pv V$ be a pseudovariety of semigroups closed under bideterministic
  product.
  The set of elements of $\Om AV$
  without good factorizations is a closed subsemigroup of $\Om AV$.
\end{Prop}

\begin{proof}
  Denote by $S$ the set of elements of $\Om AV$
  without good factorizations.

  Let $x,y\in S$.
  Suppose that $xy$ has some good factorization
  $(u,a,v)$.
  Take $b\in A$ and $z\in (\Om AV)^1$ such that $y=bz$.
  Applying Theorem~\ref{t:bidet-equidivisibility}
  to $u\cdot a\cdot v=x\cdot b\cdot z$,
  we conclude that one of the following occurs:
  \begin{enumerate}
  \item $u=x$, $a=b$ and $v=z$;
  \item $u=xbt$ and $z=tav$ for some $t\in (\Om AV)^1$;
  \item $x=uat$ and $v=tbz$ for some $t\in (\Om AV)^1$.
  \end{enumerate}
  In the first case,
  as $(u,a,v)$ is a good factorization, so is $(1,a,v)=(1,b,z)$.
  But $bz=y$ has no good factorizations, by hypothesis, and so the first case does not hold.
  If we are in the second case, then, as $bt$ is a suffix of $u$,
  we deduce from Lemma~\ref{l:shortening-good-factorizations}
  that $(bt,a,v)$ is a good factorization of $bz=y$,
  contradicting the hypothesis that $y$ has no good factorizations.
  Similarly, the third case is in contradiction with $x$ not having good factorizations. Therefore, $xy$ has no good factorization, and $S$ is a subsemigroup of $\Om AV$.

  Finally, let $(x_n)_n$ be a sequence of elements of
  $S$ converging in $\Om AV$ to~$x\in \Om AV$. Suppose that $x\notin S$.
  We may then consider a good factorization $(u,a,v)$ of~$x$.
  By Proposition~\ref{p:consequences-on-sequences},
  there are sequences $(u_n)_n$ and $(v_n)_n$
  of elements of $(\Om AV)^1$, respectively
  converging to $u$ and $v$, and there is $p$
  such that $x_n=u_nav_n$ for all $n\geq p$.
  Consider the sets
    \begin{equation*}
      J_1=\{n\geq p\mid u_na\mathrel{\mathcal R}u_n\}
      \quad\text{and}\quad
    J_2=\{n\geq p\mid av_n\mathrel{\mathcal L}v_n\}.
  \end{equation*}
  Since $x_n\in S$, every integer  $n$ greater or equal
  to $p$ belongs to $J_1\cup J_2$, and so
  at least one of the sets $J_1$ and $J_2$ is infinite.
  Suppose that $J_1$ is infinite.
  Since $\mathcal R$ is a closed relation in $(\Om AV)^1$,
  taking limits we get $ua\mathrel{\mathcal R}u$,
  contradicting that $(u,a,v)$ is a good factorization.
  Similarly, a contradiction arises if $J_2$ is infinite.
  Therefore,  $x\in S$ and so $S$ is closed.
\end{proof}

  \begin{Cor}\label{c:product-of-regulars-has-not-good-factorization}
  Let $\pv V$ be a pseudovariety of semigroups closed under bideterministic
  product. If $\pi$ is a product of regular elements of $\Om AV$,
  then $\pi$ has no good factorization.
\end{Cor}

\begin{proof}
  By Proposition~\ref{p:product-of-not-good-factorization},
  it suffices to show that an arbitrary regular element
  $\pi$ of $\Om AV$ has no good factorizations.
  Take~$s\in\Om AV$ such that $\pi=\pi s\pi$.
  Suppose there is a good factorization $(x,a,y)$ of $\pi$.
  As $\pi=x\cdot a\cdot ys\pi$,
  by Theorem~\ref{t:bidet-equidivisibility}
  one of three cases holds:
  $y=ys\pi$,
  or $xa$ is a prefix of $x$,
  or
  $ays\pi$ is a suffix of~$y$.
  The second case immediately contradicts $(x,a,y)$ being a good factorization.
  And since $ay$ is a suffix of $\pi$,
  in the first and third cases we get $y\mathrel{\mathcal L} ay$,
  also a contradiction. Hence, $\pi$ has no good factorization.
\end{proof}

The next lemma and the theorem that follows it are proved
in~\cite{Pin&Therien:1993} for the corresponding monoid versions.
We prove them in the semigroup versions with a somewhat different approach: we use pseudowords. For a semigroup~$S$, we let $Reg(S)$
be the set of regular elements of~$S$,
and let $\langle X\rangle $
be the subsemigroup of $S$ generated by
a nonempty subset $X$ of $S$.

\begin{Lemma}\label{l:product-of-regulars-in-pintherienS}
  Let $\varphi\colon A^+\to S$ be a homomorphism onto a finite semigroup,
  and let $u\in A^+$
  be such that $\varphi_{\pv{bd}}(u)$
  is a product of regular elements of $S_{\varphi}$.
  Then $u$ has no good factorizations with respect to $\varphi$.
\end{Lemma}

\begin{proof}
  Consider the unique continuous
  homomorphism $\widehat{\varphi_{\pv{bd}}}\colon\Om AS\to S_\varphi$
  extending $\varphi_{\pv{bd}}$.
  Every regular element of $S_\varphi$
  is the image by $\widehat{\varphi_{\pv{bd}}}$ of a regular element of $\Om AS$,
  and so we may take $w$ in $\langle Reg(\Om AS)\rangle$
  with $\varphi_{\pv{bd}}(u)=\widehat{\varphi_{\pv{bd}}}(w)$.
  Let $(w_n)_n$ be a sequence of words converging to $w$.
  Take the set $J$ of positive integers $n$
  such that $w_n$ has a good factorization $(x_{n},a_n,y_{n})$
  with respect to~$\varphi$.
  Suppose that $J$ is infinite.
  Let $(x,a,y)$ be an accumulation point
  of
  $(x_{n},a_n,y_{n})_{n\in J}$.
  Then $w=xay$,
  and for every $n$ in an infinite subset of $J$,
  one has $\varphi(x_{n})=\widehat\varphi(x)$,
  $a_n=a$ and $\varphi(y_{n})=\widehat\varphi(y)$.
  Therefore, $\widehat\varphi(x)<_{\mathcal R}\widehat\varphi(xa)$ and
  $\widehat\varphi(y)<_{\mathcal L}\widehat\varphi(ay)$ hold,
  thus $x<_{\mathcal R}xa$
  and $y<_{\mathcal L}ay$. Hence, $(x,a,y)$ is a good factorization
  of~$w$. But this contradicts Corollary~\ref{c:product-of-regulars-has-not-good-factorization},
  and so $J$  must be finite.
  As $\widehat{\varphi_{\pv{bd}}}(w)=\varphi_{\pv{bd}}(w_n)$
  for all large enough~$n$, we conclude that
  $\varphi_{\pv{bd}}(u)=\varphi_{\pv{bd}}(v)$ for
  some $v\in A^+$ without good factorizations
  with respect to $\varphi$. By the definition of the congruence
  $\sim_{\varphi}$, it follows that $u$ has no good factorizations
  with respect to~$\varphi$.
\end{proof}

\begin{Thm}\label{t:injective-restriction-of-the-canonical-projection}
  Let $\varphi\colon A^+\to S$ be a homomorphism onto a finite semigroup.
  Then $p_{\varphi}\colon S_{\varphi}\to S$
  restricts to an isomorphism
  $\langle Reg(S_{\varphi})\rangle\to \langle Reg(S)\rangle $.
\end{Thm}

\begin{proof}
  Since $S_\varphi$ is finite,
  we have $p_\varphi(\langle Reg(S_{\varphi})\rangle)=\langle Reg(S)\rangle$,
  so it remains to show the
  restriction is one-to-one.
  Take $s,t\in \langle Reg(S_{\varphi})\rangle$.
  Let $u,v\in A^+$ be such that $s=\varphi_{\pv{bd}}(u)$
  and $t=\varphi_{\pv{bd}}(v)$.
  By Lemma~\ref{l:product-of-regulars-in-pintherienS},
  both $u$ and $v$ have no good factorizations
  with respect to $\varphi$.
  Therefore, we have $\varphi_{\pv{bd}}(u)=\varphi_{\pv{bd}}(v)$
  if and only if $\varphi(u)=\varphi(v)$,
  that is, $s=t$ if and only if $p_\varphi(s)=p_\varphi(t)$.
\end{proof}

For the sake of conciseness, say that
a pseudoword $\pi\in \Om AS$ is \emph{$\pv V$-regular}
when $[\pi ]_{\pv V}$ is regular, and $\pi $ is  \emph{$\pv V$-multiregular} if $[\pi ]_{\pv V}$ is a finite
product of $\pv V$-regular pseudowords
(actually, one may drop the finiteness assumption in subsequent results, but the assumption is nevertheless included because of the examples we have in mind).

 The following result is a sufficient condition  to ``climb up'' a pseudoidentity from
 $\pv V$ to $\pv V_{\pv {bd}}$.

  \begin{Prop}\label{p:lift-pid}
    Let $\pv V$ be a pseudovariety of semigroups.
    If $\pi,\rho\in\Om XS$
    are
     $\pv V_{\pv {bd}}$-multiregulars,
    then
    $\pv V\models \pi=\rho$
    implies
    $\pv V_{\pv {bd}}\models \pi=\rho$.
  \end{Prop}

  \begin{proof}
    Let $\varphi\colon A^+\to S$ be a homomorphism
    onto a semigroup~$S$ of $\pv V$.
    Take an arbitrary homomorphism $\psi\colon X^+\to S_\varphi$.
    Let us show that $\widehat\psi(\pi)=\widehat\psi(\rho)$.

    Since $\widehat{\varphi_{\pv{bd}}}$ is onto,
    by the freeness of $\Om XS$
    there is a continuous homomorphism
    $\zeta\colon \Om XS\to \Om AS$ such that
    $\widehat\psi=\widehat{\varphi_{\pv{bd}}}\circ\zeta$.
    As $S\models\pi=\rho$,
    we have
    \begin{equation}\label{eq:equational-sufficient-cond-for-closure-1}
      \widehat\varphi(\zeta(\pi))=\widehat\varphi(\zeta(\rho)).
    \end{equation}
    As $S_\varphi\in \pv V_{\pv {bd}}$,
    there is a continuous homomorphism
    $\beta\colon\Om A{V_{\pv{bd}}}\to S_\varphi$
    such that
    \begin{equation}\label{eq:equational-sufficient-cond-for-closure-2}
      \widehat{\varphi_{\pv {bd}}}=\beta\circ p_{\pv V_{\pv {bd}}}.
    \end{equation}
    From the hypothesis that $\pi$ and $\rho$
    are  $\pv V_{\pv {bd}}$-multiregulars
    we get that the pseudowords
    $\zeta(\pi)$ and $\zeta(\rho)$ are also
    $\pv V_{\pv {bd}}$-multiregulars,
    and so,
    in view of~\eqref{eq:equational-sufficient-cond-for-closure-2},
    we conclude that $\widehat{\varphi_{\pv {bd}}}(\zeta(\pi))$
    and $\widehat{\varphi_{\pv {bd}}}(\zeta(\rho))$
    are both products of regular elements of
    $S_{\varphi}$.
    Then, applying Theorem~\ref{t:injective-restriction-of-the-canonical-projection},
    we obtain from equality~\eqref{eq:equational-sufficient-cond-for-closure-1}
    the equality $\widehat{\varphi_{\pv {bd}}}(\zeta(\pi))=\widehat{\varphi_{\pv {bd}}}(\zeta(\rho))$,
    that is $\widehat\psi(\pi)=\widehat\psi(\rho)$.
    Since $\psi$ is an arbitrary homomorphism from $X^+$
    into $S_{\varphi}$, we conclude that $S_{\varphi}\models \pi=\rho$.
    This shows that $\pv V_{\pv {bd}}\models \pi=\rho$.
  \end{proof}

  The characterization of
  $\overline{\pv V}$ observed in~Remark~\ref{r:tower-of-bds}
  (more precisely, the \emph{semigroup} pseudovariety
  version of Remark~\ref{r:tower-of-bds})
  and Proposition~\ref{p:lift-pid}
  allow us to deduce the following,
    with a straightforward inductive argument.
  
    \begin{Cor}\label{c:lift-pid}
    Let $\pv V$ be a pseudovariety of semigroups.
    If $\pi,\rho\in\Om XS$
    are
     $\overline{\pv V}$-multiregulars,
    then
    $\pv V\models \pi=\rho$
    implies
    $\overline{\pv V}\models \pi=\rho$.\qed
  \end{Cor}


    \begin{Defi}\label{defi:multiregularly-based}
    Let $\pv V$ and $\pv W$ be semigroup pseudovarieties
    with \mbox{$\pv V\subseteq\pv W$}.
    Say that $\pv V$ is \emph{multiregularly based in $\pv W$}
    if it has a basis $\Sigma$ of pseudoidentities
    such that, for every pseudoidentity $(\pi=\rho)$ in $\Sigma$,
    both $\pi$ and $\rho$ are
    $\pv W$-multiregular pseudowords. If $\pv W=\pv S$,
    then we just say that $\pv V$ is \emph{multiregularly based}.
  \end{Defi}

  \begin{Prop}\label{p:equational-sufficient-cond-for-closure}
    Let $\pv V$ and $\pv W$ be pseudovarieties
    of semigroups
    with $\pv V\subseteq \pv W$ and $\pv V$ is multiregularly based in $\pv W$.
    If\/ $\pv W$ is closed under bideterministic product, then
    so is $\pv V$.
  \end{Prop}

  \begin{proof}
    Let~$\Sigma$ be a basis for $\pv V$ such that, for every pseudoidentity $(\pi=\rho)$ in $\Sigma$,
    both $\pi$ and $\rho$ are
    $\pv W$-multiregulars.
    Fix an element $(\pi=\rho)$ of $\Sigma$.
    Since $\pv V\subseteq \pv W$
    and $\pv W$ is closed under bideterministic product,
    the inclusion $\pv V_{\pv {bd}}\subseteq\pv W$ holds.
    It then follows from Proposition~\ref{p:lift-pid}
    that $\pv V_{\pv {bd}}\models \pi=\rho$.
    This shows that $\pv V_{\pv {bd}}\subseteq\pv V$,
    that is, $\pv V$ is closed under bideterministic product.
  \end{proof}

   \begin{Example}\label{eg:examples-of-application-of-equational-sufficient-cond-for-closure}
   The pseudovarieties $\pv{DS}=\op ((xy)^\omega (yx)^\omega (xy)^\omega )^\omega =(xy)^\omega \cl $ and
   $\pv J=\op (xy)^\omega =(yx)^\omega , x^{\omega +1} =x^\omega \cl$ are pseudovarieties of semigroups closed under bideterministic product,
   in view of Proposition~\ref{p:equational-sufficient-cond-for-closure}.
   \end{Example}



\begin{Cor}\label{c:localization-bideterministic-product}
If\/ $\pv V$ is a pseudovariety of semigroups multiregularly based  then the pseudovariety $\Lo{V}$ is closed under bideterministic product.
\end{Cor}

\begin{proof}
Suppose that $\pv V=\op \Sigma \cl$, where $\Sigma $ is a set of pseudoidentities. For each pseudoidentity $(u=v)\in \Sigma $, consider
an alphabet $A$ that contains $c(u)\cup c(v)$ and a letter
$z\notin c(u)\cup c(v)$.
Let
$\varphi_{u,v}$ be the unique
continuous endomorphism of $\Om AS$
such that 
$\varphi(x)=z^\omega xz^\omega $
for every letter $x$ of $A$. Then the equality
$\Lo{V}= \op \varphi_{u,v}(u)=\varphi_{u,v}(v) \mid (u=v)\in \Sigma \cl$
holds.
Clearly, if $u$ and $v$ are multiregulars,
then the same happens with
$\varphi_{u,v}(u)$ and $\varphi_{u,v}(v)$.
Hence, if $\pv V$ is multiregularly based, then $\Lo{V}$ is multiregularly based and, by Proposition~\eqref{p:equational-sufficient-cond-for-closure},
$\Lo{V}$ is closed under bideterministic product.
\end{proof}

  We close this section
  applying Corollary~\ref{c:product-of-regulars-has-not-good-factorization}
  in the proof of the following technical lemma,
  to be used later on.

\begin{Lemma}\label{l:finite-prefix-exclusion}
  Let $\pv V$ be a pseudovariety of semigroups
  closed under bideterministic product.
  Consider an element $x$ of $\Om AV$
  such that $x\leq_{\R} y$
  for some regular element $y$ of $\Om AV$.
  Then, there is not a good factorization
  $(u,a,\pi)$ of $x$
  such that $u\in A^\ast$.
\end{Lemma}

\begin{proof}
  Suppose, on the contrary, that there is a good factorization
  $(u,a,\pi)$ of $x$
  such that $u\in A^\ast$.
  Let $x=yz$, with $z\in (\Om AV)^1$.
  Since $y$ is regular, it has a factorization
  $y=y_0\cdot b\cdot y_1$, such that $y_0\in A^*$, $b\in A$,
  $y_1\in \Om AV$ and $|y_0|=|u|$.
  Applying Theorem~\ref{t:bidet-equidivisibility}
  to compare the factorizations
  $(u,a,\pi)$ and $(y_0,b,y_1z)$ of $x$, and since $u$ and $y_0$ are finite words of the same length,
   we conclude that $u=y_0$, $a=b$ and $\pi=y_1z$.
  As the triple $(y_0,b,y_1z)$
  is then a good factorization of $x$,
  we may apply Lemma~\ref{l:shortening-good-factorizations} to conclude
  that the triple $(y_0,a,y_1)$ is a good factorization of $y$.
  But this contradicts Corollary~\ref{c:product-of-regulars-has-not-good-factorization}.
\end{proof}

\section{Organized factorizations}


In this section, we quickly review
the factorizations of pseudowords
as products of words and regular elements over $\pv J$,
and then proceed to an abstraction of that property.
First, it is convenient to recall the following
properties, going back to~\cite{Azevedo:1989}.
We give~\cite[Chapter 8]{Almeida:1994a} as reference.

\begin{Prop}\label{p:properties-modulo-J}
  Let $\pi,\rho\in\Om AS$. The following properties hold:
  \begin{enumerate}
  \item $\pi$ is $\pv J$-regular if and only if $\pi$ is $\pv {DS}$-regular;
    \label{item:properties-modulo-J-1}
  \item if $\pi$ and $\rho$ are $\pv J$-regular,
    then $\pv J\models \pi=\rho$ if and only if $c(\pi)=c(\rho)$;\label{item:properties-modulo-J-2}
  \item for every
    pseudovariety of semigroups $\pv V$
  such that $\pv {Sl}\subseteq \pv {V}\subseteq \pv {DS}$,
  if $\pi$ is $
  \pv V$-regular,
  then $\pv V\models \pi\rho\mathrel{\R}\pi$
  if and only if $c(\rho)\subseteq c(\pi)$.\label{item:properties-modulo-J-3}
\end{enumerate}
\end{Prop}

\begin{Defi}\label{defi:J-reduced-factorizations}
  A factorization
  $\pi=u_0\cdot\pi_1\cdot u_1\cdot\pi_2\cdots \pi_{n-1}\cdot u_{n-1}\cdot\pi_n\cdot u_n$
  of an element
  $\pi$ of $\Om AS$
  is~\emph{$\pv J$-reduced}
  if the next four conditions are satisfied:
\begin{enumerate}
\item $u_i\in A^*$ for every $i\in\{0,1,\ldots,n\}$;
\item $\pi_i$ is $\pv J$-regular for every $i\in\{1,\ldots,n\}$;
\item if $u_i=1$ and $1\leq i\leq n-1$,
  then $c(\pi_i)$ and $c(\pi_{i+1})$ are incomparable;
\item $\te 1(u_{i-1})\notin c(\pi_i)$
  and $\be 1(u_{i})\notin c(\pi_i)$, for every $i\in \{1,\ldots,n\}$.
\end{enumerate}
If, moreover, the fifth condition $u_0=u_1=\ldots=u_{n-1}=u_n=1$
is also satisfied, then we say that
$\pi=\pi_1\pi_2\cdots \pi_{n-1}\pi_n$ is a \emph{$\pv J$-reduced  multiregular element of} $\Om AS$.
\end{Defi}

\begin{Lemma}\label{l:J-reduced-are-all}
  Suppose that $\pi\in\Om AS$
  is a product of $n$ pseudowords
  that are $\pv J$-regular. Then $\pi$
  factorizes as a $\pv J$-reduced multiregular pseudoword
  $\pi_1\ldots \pi_k$ for some $k\leq n$.
 \end{Lemma}

 \begin{proof}
   Let $\pi=\pi_1\cdots\pi_n$ be a factorization into
   $\pv J$-regular pseudowords.
   We show the lemma by induction on $n$.
   The base case is trivial.
   Suppose the lemma holds for smaller values of $n$.
   If $c(\pi_i)$ and $c(\pi_{i+1})$
   are incomparable for each $i\in\{1,\ldots,n-1\}$, then the factorization is already $\pv J$-reduced.
   If, on the contrary, $c(\pi_j)$ and $c(\pi_{j+1})$
   are comparable for some~$j$, then
   $\pi_j'=\pi_j\pi_{j+1}$ is a $\pv J$-regular pseudoword
   (cf.~Proposition~\ref{p:properties-modulo-J}\eqref{item:properties-modulo-J-3})
   and $\pi=\pi_1\cdots \pi_{j-1}\pi_{j}'\pi_{j+2}\cdots \pi_n$
   is a factorization of $\pi$ into less than
   $n$ $\pv J$-regular factors. We may then apply the induction hypothesis.
 \end{proof}


As supporting references for
the next theorem,
we give~\cite[Section 4]{Almeida&Azevedo:1993} and~\cite[Theorem 8.1.11]{Almeida:1994a}.

\begin{Thm}\label{t:uniqueness-J-factorizations}
  Every element of
  $\Om AS$ has a $\pv J$-reduced factorization,
  and each $\pv J$-reduced factorization is unique modulo $\pv J$, that is,
  if
  \begin{align*}
  \pi&=u_0\cdot\pi_1\cdot u_1\cdot\pi_2\cdots \pi_{n-1}\cdot u_{n-1}\cdot\pi_n\cdot u_n,\\
  \rho&=v_0\cdot\rho_1\cdot v_1\cdot\rho_2\cdots \rho_{m-1}\cdot v_{m-1}
   \cdot \rho_m\cdot v_m,
 \end{align*}
 are $\pv J$-reduced factorizations
 such
 that $\pv J\models \pi=\rho$, then $m=n$, $u_i=v_i$
 and $\pv J\models \pi_j=\rho_j$
 for every $i\in\{0,1,\ldots,n\}$
 and $j\in \{1,\ldots,n\}$.
\end{Thm}

The following interesting observation will be used later on.

\begin{Lemma}\label{l:letters-make-products-of-regulars-fall}
  Let $\pv V$ be a semigroup pseudovariety
  in the interval $[\pv {J},\pv {DS}]$.
  Suppose that $\pi=\pi_1\cdots \pi_n \in \Om AS$
  is a $\pv J$-reduced multiregular pseudoword,
  and let $\rho\in \Om AS$.
  Then $\pv V\models \pi\rho\mathrel{\R}\pi$
  if and only if
  $c(\rho)\subseteq c(\pi_n)$.
\end{Lemma}

\begin{proof}
  As $\pv V\models \pi_n\rho\mathrel{\R}\pi_n$
  implies
  $\pv V\models \pi\rho\mathrel{\R}\pi$,
  the ``if'' part is immediate in
  view of Proposition~\ref{p:properties-modulo-J}.
  Conversely, suppose that
  $\pv V\models \pi\rho\mathrel{\R}\pi$,
  and let $x$ be such that $\pv V\models \pi=\pi \rho x$.
  We then have
  $\pv V\models \pi_1\cdots\pi_n=\pi_1\cdots\pi_n(\rho x)^\omega$.
  By the uniqueness of $\pv J$-reduced factorizations
  (Theorem~\ref{t:uniqueness-J-factorizations}),
  the sets $c(\pi_n)$ and $c((\rho x)^\omega)$
   must be comparable,
   and so $\pi'_n=\pi_n(\rho x)^\omega$
   is $\pv V$-regular.
   Again by the uniqueness of $\pv J$-reduced
   factorizations and by Lemma~\ref{l:J-reduced-are-all} ,
   the factorization
   $\pi_1\cdots\pi_{n-1}\pi'_n$
   must be $\pv J$-reduced,
   and moreover $\pv J\models \pi_n=\pi_n'$.
   In particular, we have $c(\rho)\subseteq c(\pi_n)$.
\end{proof}


Next is an abstraction of
some features of being $\pv J$-reduced.

  \begin{Defi}\label{defi:factorizations-levels}

    Let us consider a factorization of $\pi\in \Om AS$ of the form
    \begin{equation}\label{eq:factorizations-levels-1}
      \pi=
      u_0\cdot\pi_1\cdot u_1\cdot\pi_2\cdots \pi_{n-1}\cdot u_{n-1}
      \cdot
      \pi_n\cdot u_n,
    \end{equation}
    for some $n\geq 0$,
    such that $u_0,u_n\in A^\ast$, $u_i\in A^+$ when $1\leq i\leq n-1$,
    and $\pi_i\notin A^+$ when $1\leq i\leq n$.

    Let $\pv V$ be a semigroup pseudovariety.
    The factorization~\eqref{eq:factorizations-levels-1}
    is
    \emph{multiregularly organized in $\pv V$} if  the pseudowords $\pi_i$, $1\leq i\leq n$, are $\pv V$-multiregular.
    The same factorization~\eqref{eq:factorizations-levels-1}
    is
    \emph{organized with short $\pv V$-breaks}
    if  the following conditions hold:
    \begin{enumerate}
    [label=(SB.\arabic*),series=axioms]
    \item $[\pi_i\,\be 1(u_i)]_{\pv V}<_{\R}[\pi_i]_{\pv V}$ when
      $1\leq i\leq n-1$;\label{item:SBR1}
    \item $[\te 1(u_{i})\,\pi_{i+1}]_{\pv V}<_{\L}[\pi_{i+1}]_{\pv V}$
      when $1\leq i\leq n-1$;\label{item:SBR2}
    \item if $u_0\neq 1$ then $[\te 1(u_{0})\,\pi_1]_{\pv V}<_{\L}[\pi_1]_{\pv V}$;\label{item:SBR3}
    \item if $u_n\neq 1$ then $[\pi_n\,\be 1(u_n)]_{\pv V}<_{\R}[\pi_n]_{\pv V}$.\label{item:SBR4}
    \end{enumerate}
    Finally, the factorization~\eqref{eq:factorizations-levels-1}
    is
    \emph{organized with long $\pv V$-breaks}
    if  the following conditions hold:
    \begin{enumerate}
      [label=(LB.\arabic*),series=axioms]
    \item $[u_0\pi_1u_1\pi_2\cdots\pi_{i-1}u_{i-1}\pi_i\,\be 1(u_i)]_{\pv V}
         <_{\R}
         [u_0\pi_1u_1\pi_2\cdots\pi_{i-1}u_{i-1}\pi_i]_{\pv V}$
         for each $i\in \{1,\ldots,n-1\}$;\label{item:LBR1}
       \item $[\te 1(u_{i})\,\pi_{i+1}u_{i+1}\pi_{i+2}\cdots u_{n-1}\pi_nu_n]_{\pv V}<_{\L}
         [\pi_{i+1}u_{i+1}\pi_{i+2}\cdots u_{n-1}\pi_nu_n]_{\pv V}$
         for each $i\in \{1,\ldots,n-1\}$;\label{item:LBR2}
    \item if $u_0\neq 1$ then
      $ [\te 1(u_{0})\,\pi_1u_1\pi_2\cdots u_{n-1}\pi_nu_n]_{\pv V}<_{\L}
      [\pi_1u_1\pi_2\cdots u_{n-1}\pi_nu_n]_{\pv V}$;\label{item:LBR3}
    \item if $u_n\neq 1$ then
      $[u_0\pi_1u_1\pi_2\cdots u_{n-1}\pi_n\,\be 1(u_n)]_{\pv V}
      <_{\R}[u_0\pi_1u_1\pi_2\cdots u_{n-1}\pi_{n}]_{\pv V}$.\label{item:LBR4}
    \end{enumerate}
  \end{Defi}

  \begin{Remark}\label{r:long-implies-short}
    A factorization~\eqref{eq:factorizations-levels-1} that is organized with long $\pv V$-breaks
    is a factorization that is organized with short $\pv V$-breaks.
  \end{Remark}

    \begin{Defi}\label{defi:organizing-pseudovariety}
    A semigroup pseudovariety $\pv W$ is \emph{organizing}
    if, for every alphabet~$A$, every $\pi\in\Om AS$ is multiregularly organized in $\pv W$.
  \end{Defi}

  \begin{Example}\label{eg:organizing-pseudovarieties}
    The pseudovarieties $\pv {DS}$
    and $\pv {DS}\cap\pv{ECom}$
    are organizing pseudovarieties,
    the former by Theorem~\ref{t:uniqueness-J-factorizations}
    (and Proposition~\ref{p:properties-modulo-J}\eqref{item:properties-modulo-J-1}), the latter by an analog result of Almeida and Weil (Proposition 3.7
    in~\cite{Almeida&Weil:1994c}).
\end{Example}


            \begin{Prop}\label{p:reduction-of-pseudoword}
              Let $\pv V$ be a pseudovariety of semigroups.
        Suppose that $\pi$
        is an element of $\Om AS$
        with a factorization multiregularly organized in
        a semigroup pseudovariety $\pv W$.
        Then there are $\pi'\in\Om AS$ and a set $\Upsilon_\pi$
        of pseudoidentities over~$A$
        satisfying the following conditions:
        \begin{enumerate}
        \item
          $\pv V\subseteq\op\Upsilon_\pi\cl\subseteq \op\pi=\pi'\cl$;\label{item:reduction-of-pseudoword-1}
        \item $\pi'$ has a factorization multiregularly organized in $\pv W$ and with short
          \mbox{$\pv V$-breaks};\label{item:reduction-of-pseudoword-3}
        \item if $(\varpi=\varrho)$ belongs to $\Upsilon_\pi$,
          then $\varpi$ and $\varrho$\label{item:reduction-of-pseudoword-4}
          are  $\pv W$-multiregulars.
        \end{enumerate}
      \end{Prop}

      \begin{proof}
        By hypothesis, there is a factorization
        \begin{equation*}
          \pi=
      u_0\cdot\pi_1\cdot u_1\cdot\pi_2\cdots \pi_{n-1}\cdot u_{n-1}
      \cdot
      \pi_n\cdot u_n
    \end{equation*}
    which is multiregularly organized in $\pv W$.
          Let $s=\sum_{i=0}^n|u_i|$ and $\sigma=n+s$.
          We show the proposition by induction
          on $\sigma\geq 0$.
    If $n=0$ or $s=0$ (the latter implies $n=1$), then just take $\pi'=\pi$
    and $\Upsilon_\pi=\emptyset$.
    In particular, this shows the initial step of the induction.

    Suppose the proposition
    holds for smaller values of $\sigma$
    and that $n>0$ and $s>0$.
     Let $J$ be the set of integers $i$ in $\{1,\ldots,n\}$
     such that
     $[\pi_i\,\be 1(u_i)]_{\pv V}\mathrel{\R}[\pi_i]_{\pv V}$
     and $u_i\neq 1$,
     or such that
     $[\te 1(u_{i-1})\,\pi_{i}]_{\pv V}\mathrel{\L}[\pi_{i}]_{\pv V}$
     and $u_{i-1}\neq 1$.
    In fact, one always has $u_i\neq 1$, except, perhaps, when
    $i=0$ or $i=n$.
    If $J=\emptyset$, then we just take $\pi'=\pi$
    and $\Upsilon_\pi=\emptyset$.
    Suppose that $J\neq\emptyset$,
    and let~$j\in J$.
    Without loss of generality, we assume
    that
    $[\pi_j\,\be 1(u_j)]_{\pv V}\mathrel{\R}[\pi_j]_{\pv V}$
     and $u_j\neq 1$.
    Let $a=\be 1(u_j)$, and let $v\in A^\ast$ be such that $u_j=av$.
    Then, there is $x\in\Om AS$ such that
    \begin{equation}\label{eq:reduction-of-pseudoword-3}
      [\pi_j]_{\pv V}=[\pi_jax]_{\pv V}=[\pi_j(ax)^\omega]_{\pv V}.
    \end{equation}
    Let $\rho=\pi_j(ax)^\omega a$.
    Since $(ax)^\omega a$
    is a regular element of $\Om AS$ and
    $[\pi_j]_{\pv W}$ is $\pv W$-multiregular,
    we know that $[\rho]_{\pv W}$ is also $\pv W$-multiregular. Consider the pseudoword
    $\bar\pi$ with factorization
        \begin{equation*}
          \bar\pi=
          u_0\cdot\pi_1\cdot u_1\cdot\pi_2\cdots
          \pi_{j-1}\cdot u_{j-1}\cdot \rho\cdot v\cdot \pi_{j+1}
          \cdots
          \pi_{n-1}\cdot u_{n-1}
      \cdot
      \pi_n\cdot u_n.
    \end{equation*}
    If $v\neq 1$,
    then this factorization is multiregularly organized in $\pv W$.
    If $v=1$, then we also obtain a factorization
    of $\bar\pi$
    which is multiregularly organized in $\pv W$
    by ``gluing'' $\rho$ and $\pi_{j+1}$.
    In any case, we produce a factorization of
    $\bar\pi$ which is multiregularly organized in $\pv W$
    and with smaller value for $\sigma$ than that we had in~$\pi$.
    We may therefore apply the induction hypothesis
    to obtain a pseudoword $\pi'\in\Om AS$
    and a set $\Upsilon_{\bar\pi}$
    of pseudoidentities over $A$
    satisfying the following conditions:
        \begin{enumerate}
        \item $\pv V\subseteq\op\Upsilon_{\bar\pi}\cl\subseteq \op\bar\pi=\pi'\cl$;
        \item $\pi'$ has a factorization multiregularly organized in $\pv W$ and with short
          $\pv V$-breaks;
        \item if $(\varpi=\varrho)$ belongs to $\Upsilon_{\bar\pi}$,
          then $\varpi$ and $\varrho$
          are $\pv W$-multiregulars.
        \end{enumerate}
        Let
          $\Upsilon_{\pi}=\Upsilon_{\bar\pi}\cup \{(\pi_j=\pi_j(ax)^\omega)\}$.
        By~\eqref{eq:reduction-of-pseudoword-3},
        we know that
        $\pv V\subseteq \op \Upsilon_{\pi}\cl$.
        Moreover, as $[\pi_j]_{\pv W}$
        and $[\pi_j(ax)^\omega]_{\pv W}$
        are both $\pv W$-multiregulars,
        we immediately get that Condition~\eqref{item:reduction-of-pseudoword-4}
        in the statement of the proposition holds for $\Upsilon_\pi$.
        Consider a semigroup $S$ in $\op\Upsilon_\pi\cl$.
        Then $S$ belongs to $\op\Upsilon_{\bar\pi}\cl$,
        whence
        \begin{equation*}
          S\models \bar\pi=\pi'.
        \end{equation*}
        On the other hand,
        we also have $S\models \pi_j=\pi_j(ax)^\omega$,
        which implies that
        \begin{equation*}
          S\models \rho\cdot v=\pi_j(ax)^\omega av=\pi_ju_j.
        \end{equation*}
        We immediately conclude that $S\models \pi=\bar\pi=\pi'$.
        We have therefore showed that $\op\Upsilon_\pi\cl\subseteq \op\pi=\pi'\cl$,
        concluding the inductive step of the proof.
      \end{proof}









      \begin{Cor}\label{c:basis-of-nice-pseudoidentities}
        Let $\pv V$ and $\pv W$ be pseudovarieties
        of semigroups, with $\pv W$ being an organizing
      pseudovariety.
      Then $\pv V$ has a basis of pseudoidentities with factorizations multiregularly organized in
      $\pv W$ and with short $\pv V$-breaks.
    \end{Cor}

    \begin{proof}
      Let $\Sigma$ be a basis of pseudoidentities
      for $\pv V$.
      For each $(\pi=\rho)\in\Sigma$,
      let
      \begin{equation*}
        \Gamma_{(\pi=\rho)}=
        \Upsilon_\pi\cup\Upsilon_\rho\cup\{(\pi'=\rho')\},
      \end{equation*}
      where we follow the definition included in Proposition~\ref{p:reduction-of-pseudoword}.
      Consider the union $\Gamma=\bigcup_{(\pi=\rho)\in\Sigma}\Gamma_{(\pi=\rho)}$.
      It suffices to show that
      $\pv V=\op\Gamma\cl$.

      Suppose that $S\in\pv V$.
      For each $\pi$,
      one has
      $S\models \Upsilon_\pi$
      and $S\models \pi=\pi'$.
      In particular, $S\models \pi=\rho$
      implies $S\models \pi'=\rho'$.
      We conclude that if $(\pi=\rho)\in\Sigma$,
      then $S\models \Gamma_{(\pi=\rho)}$ holds.
      This establishes the inclusion
      $\pv V\subseteq \op\Gamma\cl$.

      Conversely, let $S\in \op\Gamma\cl$.
      Take $(\pi=\rho)\in\Sigma$.
      Because $S\models \Upsilon_\pi$, we have
      $S\models \pi=\pi'$.
      Similarly, $S\models \rho=\rho'$ holds.
      But $S\in \op\Gamma\cl$ also implies
      $S\models \pi'=\rho'$,
      and so, all together, we get $S\models \pi=\rho$.
      This shows $\op\Gamma\cl\subseteq \pv V$.
    \end{proof}

  \section{Breaking factorizations assuming bideterministic closure}
  \label{sec:break-fact-assum}
  
  In this section we
  see how closure under bidetermistic product
  allows us to decompose the pseudoidentities
  found in Corollary~\ref{c:basis-of-nice-pseudoidentities}
  into pieces involving only
  products of regular elements over the organizing pseudovariety.

  \begin{Prop}\label{p:weak-organized-is-strong-organized}
    Consider a semigroup pseudovariety
    closed under bideterministic product.
  Let $\pi\in\Om AS$.
  A factorization
  of~$\pi$
  is organized with short $\pv V$-breaks
  if and only if it is organized with long $\pv V$-breaks.
  \end{Prop}

  \begin{proof}
    Recall that the ``if'' part of the theorem
    is immediate (Remark~\ref{r:long-implies-short}).

    Conversely, consider a factorization
     \begin{equation}\label{eq:weak-organized-is-strong-organized-0}
       \pi=
       u_0\cdot\pi_1\cdot u_1\cdot\pi_2\cdots \pi_{n-1}\cdot u_{n-1}
       \cdot
       \pi_n\cdot u_n,
     \end{equation}
     that is organized with short $\pv V$-breaks.
     We prove by induction on $n\geq 0$ that it is
     organized with long $\pv V$-breaks.
     The base case $n=0$ holds trivially. Suppose that
     $n>0$ and that  the theorem holds for smaller values of $n$.

     We show that Conditions~\ref{item:LBR1} and~\ref{item:LBR4}
     in Definition~\ref{defi:factorizations-levels}
     hold for~\eqref{eq:weak-organized-is-strong-organized-0}.
     Since the factorization
     $u_0\cdot\pi_1\cdot u_1\cdot\pi_2\cdots \pi_{n-1}\cdot u_{n-1}$
     is clearly organized with short $\pv V$-breaks,
     it is, by the induction hypothesis, organized with long $\pv V$-breaks.
     Therefore, to establish Conditions~\ref{item:LBR1} and \ref{item:LBR4}
     in Definition~\ref{defi:factorizations-levels}
     for the factorization~\eqref{eq:weak-organized-is-strong-organized-0},
     it only remains to show that
     it is impossible to have $u_n\neq 1$ and
     \begin{equation}\label{eq:weak-organized-is-strong-organized-2}
       [u_0\pi_1u_1\pi_2\cdots\pi_{n-1}u_{n-1}\pi_n\,\be 1(u_n)]_{\pv V}
      \mathrel{\R}[u_0\pi_1u_1\pi_2\cdots\pi_{n-1}u_{n-1}\pi_{n}]_{\pv V}.
     \end{equation}
     Suppose, on the contrary, that that is possible.
     Take $\rho=u_0\pi_1u_1\pi_2\cdots\pi_{n-1}$ and
     $u_n=bw$, with $b\in A$ and $w\in A^*$.
     Then, by Lemma~\ref{l:fall}, for every sufficiently
     large positive integer $k$,
     the sets
     \begin{equation*}
       \Bigr[B\Bigl([\rho]_{\pv V},\frac{1}{k}\Bigr)\cap A^\ast\Bigl]\cdot u_{n-1}
       \quad\text{and}\quad
       \Bigr[B\Bigl([\pi_n]_{\pv V},\frac{1}{k}\Bigr)\cap A^\ast\Bigl]\cdot b
     \end{equation*}
     are prefix codes.
     Therefore, for every sufficiently large $k$,
     their product
     \begin{equation*}
       P_k=\Bigr[B\Bigl([\rho]_{\pv V},\frac{1}{k}\Bigr)\cap A^\ast\Bigl]
       \cdot u_{n-1}\cdot
       \Bigr[B\Bigl([\pi_n]_{\pv V},\frac{1}{k}\Bigr)\cap A^\ast\Bigl]
       \cdot b
     \end{equation*}
     is a prefix code.
     Following the notation
     of Proposition~\ref{p:a-sort-of-open-multiplication-+},
     note that
     \begin{equation*}
       P_k
       =\Bigl[L_k\Bigl([\rho]_{\pv V},u_{n-1},[\pi_n]_{\pv V}\Bigr)\cap A^+\Bigr]\cdot b.
     \end{equation*}
     On the other hand, since
     $\bigl([\rho]_{\pv V},u_{n-1},[\pi_n]_\pv V\bigr)$
     is a $+$-good factorization,
     the language
       $L_k\bigl([\rho]_{\pv V},u_{n-1},[\pi_n]_{\pv V}\bigr)\cap A^+$
     is $\pv V$-recognizable for every sufficiently large $k$
     (cf.~Proposition~\ref{p:a-sort-of-open-multiplication-+}).
     Therefore, again applying the hypothesis
     that~$\pv V$ is closed under bideterministic product,
     we  conclude that $P_k$ is a $\pv V$-recognizable language,
     and so the closure $\overline{P_k}^{\pv V}$ of $P_k$ in $\Om AV$
     is a clopen
     subset of $\Om AV$.
     Since
     \begin{equation*}
       \overline{P_k}^\pv V=
       B\Bigl([\rho]_{\pv V},\frac{1}{k}\Bigr)\cdot u_{n-1}\cdot
       B\Bigl([\pi_n]_{\pv V},\frac{1}{k}\Bigr)\cdot b,
     \end{equation*}
     we have
     $[\rho \cdot u_{n-1}\cdot \pi_n\cdot b]_{\pv V}\in \overline{P_k}^\pv V$.
     As we are assuming that~\eqref{eq:weak-organized-is-strong-organized-2},
     holds, there is $x\in\Om AS$ such that
     $[\rho \cdot u_{n-1}\cdot \pi_n]_{\pv V}=
       [\rho \cdot u_{n-1}\cdot \pi_n\cdot b x]_{\pv V}$
     and thus
          \begin{equation}\label{eq:weak-organized-is-strong-organized-3}
       [\rho \cdot u_{n-1}\cdot \pi_n\cdot b]_{\pv V}=
       [\rho \cdot u_{n-1}\cdot \pi_n\cdot b xb]_{\pv V}.
     \end{equation}
     Let $(\rho_m)_m$, $(\pi_{n,m})_m$ and $(x_m)_m$
     be sequences of elements of $A^+$ respectively
     converging in $\Om AS$ to $\rho$, $\pi_n$ and $x$,
     and let $w_m=\rho_m\cdot u_{n-1}\cdot \pi_{n,m}\cdot bx_mb$.
     Note that, for every sufficiently large $m$,
     one has $\rho_m\in B\Bigl([\rho]_{\pv V},\frac{1}{k}\Bigr)$
     and also $\pi_{n,m}\in B\Bigl([\pi_n]_{\pv V},\frac{1}{k}\Bigr)$,
     thus $\rho_m\cdot u_{n-1}\cdot \pi_{n,m}\cdot b\in P_k$
     and $w_m\in P_k\cdot A^+$.
     On the other hand, by~\eqref{eq:weak-organized-is-strong-organized-3},
     the sequence of words $(w_m)_m$ converges
     in $\Om AV$ to $[\rho \cdot u_{n-1}\cdot \pi_nb]_{\pv V}$.
     Since the latter has $\overline{P_k}^\pv V$ as a neighborhood,
     we conclude that for sufficiently large $m$
     the word $w_m$ belongs to the intersection
     $P_k\cap P_k A^+$.
     But this contradicts
     $P_k$ being a prefix code. Therefore, in order to avoid this contradiction,
     one must not have~\eqref{eq:weak-organized-is-strong-organized-2}
     whenever $u_n\neq 1$.

      We established
      Conditions~\ref{item:LBR1} and~\ref{item:LBR4} in Definition~\ref{defi:factorizations-levels}
      for the factorization~\eqref{eq:weak-organized-is-strong-organized-0}.
     Symmetrically,
     Conditions~\ref{item:LBR2} and~\ref{item:LBR3}
     hold for the same factorization. This concludes the inductive step.
   \end{proof}

We are now ready for showing the central result of this paper.

\begin{Thm}\label{t:cut-nice-factorizations}
  Let $\pv V$ be a semigroup pseudovariety
  closed under bideterministic product.
  Take $\pi,\rho\in\Om AS$ such that
  $\pv V\models \pi=\rho$. If
   \begin{align}
   \pi&= u_0\cdot\pi_1\cdot u_1\cdot\pi_2\cdots \pi_{n-1}\cdot u_{n-1}
   \cdot \pi_n\cdot u_n,\label{eq:cut-nice-factorizations-1.1}
   \\
   \rho&= v_0\cdot\rho_1\cdot v_1\cdot\rho_2\cdots \rho_{m-1}\cdot v_{m-1}
   \cdot \rho_m\cdot v_m,\label{eq:cut-nice-factorizations-1.2}
    \end{align}
    are factorizations
    multiregularly organized in $\pv V$
    and with short $\pv V$-breaks,
    then~\mbox{$n=m$}, $u_i=v_i$ and $\pv V\models \pi_j=\rho_j$ for every
    $i\in \{0,1,\ldots,n\}$ and $j\in \{1,\ldots,n\}$.
  \end{Thm}

  \begin{proof}
    If $n=0$, then $\pi\in A^+$. Since $\pv N\models \pi=\rho$,
    we then must have $\pi=\rho$, $m=0$ and $u_0=\pi=\rho=v_0$.

    Suppose that $n,m\geq 1$.
    We prove the theorem by induction on
    \begin{equation*}
      s=n+m+\sum_{i=0}^n|u_i|+\sum_{j=0}^m|v_j|.
    \end{equation*}
    Note that $s\geq n+m\geq 2$.
    If $s=2$, then $n=m=1$ and $u_i=v_j=1$ for every possible $i,j$.
    Therefore, $\pv V\models \pi =\pi _1=\rho _1=\rho$ and the theorem
    holds in the base case $s=2$.

    Suppose that $s>2$, and suppose that the theorem holds for smaller values
    of $s$. We consider the (possibly empty) pseudowords
    \begin{equation*}
   \bar\pi= \pi_2\cdots \pi_{n-1}\cdot u_{n-1}
   \cdot \pi_n\cdot u_n
   \qquad\text{and}\qquad
   \bar\rho= \rho_2\cdots \rho_{m-1}\cdot v_{m-1}
   \cdot \rho_m\cdot v_m.
 \end{equation*}
 Note that, according to Proposition~\ref{p:weak-organized-is-strong-organized},
 both factorizations~\eqref{eq:cut-nice-factorizations-1.1}
 and~\eqref{eq:cut-nice-factorizations-1.2}
 are multiregularly organized with long $\pv V$-breaks. Therefore,
 if $u_0=u'a$, with $a\in A$ and $u'\in A^*$, then
 $(u',a,[\pi_1u_1\bar\pi]_{\pv V})$
 is a good factorization,
 and if
 $v_0=v'b$, with $b\in A$ and $v'\in A^*$, then
 $(v',b,[\rho_1u_1\bar\rho]_{\pv V})$
 is a good factorization.
 Then, taking into account
 that both $\pi_1$ and $\rho_1$
 have regular elements of
 $\Om AV$ as prefixes (they are $\pv V$-multiregulars),
 applying Lemma~\ref{l:finite-prefix-exclusion}
 we conclude that $u_0=1$ if and only if $v_0=1$.

    Suppose that $u_0\neq 1$ and $v_0\neq 1$.
    Without loss of generality, assume that $|u_0|\leq |v_0|$.
 Consider the factorization
 $u_0=u'a$ such that $a\in A$ and $u'\in A^\ast$,
 and the factorization $v_0=v'bv''$ such that $b\in A$, $v',v''\in A^*$ and $|v'|=|u'|$.
 Since $u'$ and $v'$ are finite words,
 when we use Theorem~\ref{t:bidet-equidivisibility}
 to compare the factorizations
 $(u',a,[\pi_1u_1\bar\pi]_{\pv V})$
 and
 $(v',b,[v''\rho_1v_1\bar\rho]_{\pv V})$
 of $[\pi]_{\pv V}=[\rho]_{\pv V}$,
 the first of which is a good one,
 the only possibility
 is that $u'=v'$, $a=b$, and
 \begin{equation}\label{eq:cut-indunction-1}
      \pv V\models
      \pi_1\cdot u_1\cdot\pi_2\cdots \pi_{n-1}\cdot u_{n-1}
      \cdot \pi_n\cdot u_n
      =v''\cdot\rho_1\cdot v_1\cdot\rho_2\cdots \rho_{n-1}\cdot v_{n-1}
   \cdot \rho_m\cdot v_m.
 \end{equation}
 Since
 \begin{align*}
   &(n+m+\sum_{i=0}^n|u_i|+\sum_{j=0}^m|v_j|)-
   (n+m+\sum_{i=1}^n|u_i|+|v''|+\sum_{j=1}^m|v_j|)=\\
   =&|u_0|+|v_0|-|v''|=|u_0|+|v'|+1>0,
 \end{align*}
 we may apply in~\eqref{eq:cut-indunction-1} the induction hypothesis,
 from which we conclude that $n=m$, $v''=1$, $u_i=v_i$ and $\pv V\models \pi_j=\rho_j$ for every
 $i\in \{0,1,\ldots,n\}$ and $j\in \{1,\ldots,n\}$.
 Therefore, we may suppose from hereon that $u_0=v_0=1$.

 Suppose that $u_1=1$ also holds (the case $v_1=1$ is symmetric).
 Note that then one has $n=1$.
 Moreover, $[\pi]_{\pv V}=[\rho]_{\pv V}$ has no good factorizations
 (cf.~Corollary~\ref{c:product-of-regulars-has-not-good-factorization}),
 whence it has no $+$-good factorization
 (cf.~Lemma~\ref{l:good-and-+-good}).
 The latter implies $m=1$, $v_0=v_1=1$.
 Hence, the theorem holds in this~case.

    Finally, we suppose that $u_1\neq 1$ and $v_1\neq 1$.
    Consider factorizations
    $u_1=c\bar u_1$ with $c\in A$ and $\bar u_1\in A^*$,
    and
    $v_1=d\bar v_1$ with $d\in A$ and $\bar v_1\in A^*$.
    Compare the factorizations
    $([\pi_1]_{\pv V},c,[\bar u_1\bar \pi]_{\pv V})$
    and
    $([\rho_1]_{\pv V},d,[\bar v_1\bar \rho]_{\pv V})$.
    They are good, by Lemma~\ref{l:good-and-+-good}.
    By Theorem~\ref{t:bidet-equidivisibility},
    one of the following cases
    holds:
    \begin{enumerate}
    \item $[\pi_1]_{\pv V}=[\rho_1]_{\pv V}$,
      $c=d$ and $[\bar u_1\bar \pi]_{\pv V}=[\bar v_1\bar \rho]_{\pv V}$;\label{item:1}
    \item $[\pi_1]_{\pv V}=[\rho_1dt]_{\pv V}$
      and $[tc\bar u_1\bar \pi ]_{\pv V}=[\bar v_1\bar \rho]_{\pv V}$,
      for some $t\in (\Om AS)^1$;\label{item:2}
   \item $[\pi_1ct]_{\pv V}=[\rho_1]_{\pv V}$
     and $[\bar u_1\bar \pi ]_{\pv V}=[td\bar v_1\bar \rho]_{\pv V}$,
     for some $t\in (\Om AS)^1$.\label{item:3}
    \end{enumerate}
    Suppose Case~\eqref{item:2} holds.
    Since
    $([\rho_1]_{\pv V},d,[tc\bar u_1\bar \pi]_{\pv V})
      =([\rho_1]_{\pv V},d,[\bar v_1\bar \rho]_{\pv V})$
    is a good factorization
    of $[\rho]_{\pv V}$,
    it follows from Lemma~\ref{l:shortening-good-factorizations} that
    $([\rho_1],d,[t]_{\pv V})$
    is a good factorization of
    $[\rho_1dt]_{\pv V}=[\pi_1]_{\pv V}$.
    But this is impossible in view of Corollary~\ref{c:product-of-regulars-has-not-good-factorization},
    because $[\pi_1]_{\pv V}$
    is a product of regular elements.
    Therefore, Case~\eqref{item:2} does not hold.
    By symmetry, we conclude that Case~\eqref{item:3} is also impossible,
    and therefore only Case~\eqref{item:1} holds.
    We may then apply the induction hypothesis
    to
    \begin{equation*}
      \pv V\models
      \bar u_1\cdot\pi_2
      \cdot u_2\cdot\pi_3
      \cdots \pi_{n-1}\cdot u_{n-1}
      \cdot \pi_n\cdot u_n
      =\bar v_1\cdot\rho_2\cdot v_2\cdot\rho_3\cdots \rho_{m-1}\cdot v_{m-1}
   \cdot \rho_m\cdot v_m
    \end{equation*}
    to obtain
    $\bar u_1=\bar v_1$ (and so $u_1=v_1$),
    $n=m$,  $u_i=v_i$ and $\pv V\models \pi_j=\rho_j$ for every
    $i\in \{0,1,\ldots,n\}$ and $j\in \{1,\ldots,n\}$.

    This exhausts all possibles cases
    to consider in the inductive step.
  \end{proof}


  \begin{Cor}\label{c:basis-of-products-of-regular-elements}
    Let $\pv V$ be a pseudovariety of semigroups
    closed under bideterministic product.
    Suppose that $\pv W$ is an organizing pseudovariety
    containing~$\pv V$.
    Then $\pv V$ is multiregularly based in $\pv W$.
  \end{Cor}

  \begin{proof}
    By Corollary~\ref{c:basis-of-nice-pseudoidentities},
    $\pv V$ has a basis $\Sigma$ of
    pseudoidentities multiregularly organized in $\pv W$ (whence in $\pv V$)  and with short $\pv V$-breaks.
    Let $(\pi=\rho)\in \Sigma$.
    Suppose that
    the next factorizations
    are multiregularly organized in $\pv W$
  and have short $\pv V$-breaks:
        \begin{equation*}
      \pi= u_0\cdot\pi_1\cdot u_1\cdot\pi_2\cdots \cdot u_{n-1}
   \cdot \pi_n\cdot u_n
   \quad\text{and}\quad
   \rho= v_0\cdot\rho_1\cdot v_1\cdot\rho_2\cdots \cdot v_{m-1}
   \cdot \rho_m\cdot v_m.
 \end{equation*}
    By Theorem~\ref{t:cut-nice-factorizations},
    we know that  $n=m$, $u_i=v_i$ and $\pv V\models \pi_j=\rho_j$ for every
    $i\in \{0,1,\ldots,n\}$ and $j\in \{1,\ldots,n\}$.
    The integer $n$ depends on $(\pi=\rho)$, and for that
    reason we denote it by $n_{(\pi=\rho)}$.
    For each $(\pi=\rho)\in\Sigma$,
    let $\Gamma_{(\pi=\rho)}=\{(\pi_j=\rho_j)\mid 1\leq j\leq n_{(\pi=\rho)}\}$.
    Consider the set of pseudoidentities
    $\Gamma=\bigcup_{(\pi=\rho)\in\Sigma}\Gamma_{(\pi=\rho)}$.
    To conclude the proof, it suffices to show
    that $\pv V=\op\Gamma\cl$. We already saw that $\pv V\models \Gamma$.
    Conversely, suppose that $S$ is a semigroup
    such that $S\models \Gamma$.
    Fix a pseudoidentity $(\pi=\rho)\in\Sigma$.
    For each $j\in\{1,\ldots,n_{(\pi=\rho)}\}$,
    we have  $S\models \pi_j=\rho_j$.
    This clearly implies $S\models \pi=\rho$,
    in view of the factorizations of $\pi$ and $\rho$
    with which we are working with.
    Hence, we have $S\models \Sigma$, that is,  $S\in \pv V$.
    This concludes the proof that $\pv V=\op\Gamma\cl$.
  \end{proof}

  We next highlight the case where the organizing pseudovariety is $\pv {DS}$.

  \begin{Thm}\label{t:equational-characterization-bidet-J-DS}
    Suppose that $\pv V$
    is a semigroup
    pseudovariety in the interval
    $[\pv {Sl},\pv {DS}]$. The following conditions are equivalent:
    \begin{enumerate}
    \item $\pv V$ is closed under bideterministic product;\label{item:equational-characterization-bidet-J-DS-1}
    \item $\pv V$ is multiregularly based in $\pv {DS}$;\label{item:equational-characterization-bidet-J-DS-2}
    \item $\pv V$ is multiregularly based.\label{item:equational-characterization-bidet-J-DS-3}
    \end{enumerate}
\end{Thm}

\begin{proof}
  \eqref{item:equational-characterization-bidet-J-DS-1}$\Rightarrow$\eqref{item:equational-characterization-bidet-J-DS-2}:
  This is a direct application
  of Corollary~\ref{c:basis-of-products-of-regular-elements},
  in view of the fact that
  $\pv {DS}$ is an organizing pseudovariety.

  \eqref{item:equational-characterization-bidet-J-DS-2}$\Rightarrow$\eqref{item:equational-characterization-bidet-J-DS-3}:
  If $\pi=\pi_1\cdots\pi_n$ is a 
   $\pv {DS}$-multiregular pseudoword,
  then the product $\pi'=\pi_1^{\omega+1}\cdots\pi_n^{\omega+1}$
  is a multiregular pseudoword
  for which we have $\pv {DS}\models \pi=\pi'$.
  Therefore, if $\Sigma$ is a basis for
  $\pv V$ formed by pseudoidentities
  between $\pv {DS}$-multiregular pseudowords,
  then
  \begin{equation*}
    \Sigma'=\{(\pi'=\rho')\mid (\pi=\rho)\in\Sigma\}
    \cup
    \{((xy)^\omega(yx)^\omega(xy)^\omega)^\omega=(xy)^\omega\}
  \end{equation*}
  is a basis for $\pv V$,
  comprised solely by pseudoidentities
  between products of $\pv S$-regular pseudowords.

  \eqref{item:equational-characterization-bidet-J-DS-3}$\Rightarrow$\eqref{item:equational-characterization-bidet-J-DS-1}:
  It follows from Proposition~\ref{p:equational-sufficient-cond-for-closure}.
 \end{proof}

\section{Factorizations in the global}
\label{sec:fact-glob}
 

\subsection{Semigroupoids}\label{subsec-preamble-semigroupoids}


For the reader
to situate himself better, we give some notation and recall some facts
on semigroupoids.
We refer to~\cite{Jones:1996,Almeida&Weil:1996,Rhodes&Steinberg:2009qt}.

A \emph{semigroupoid} $S$ is a graph $V(S)\cup E(S)$ endowed with two
operations $\alpha,\ \omega \colon E(S)\to V(S)$ which give respectively the
\emph{beginning} and \emph{end} vertices of each edge, and a partial
associative multiplication on $E(S)$ given by: for $s,t\in E$, $st$ is
defined if and only if $\omega (s)=\alpha (t)$ and, then,
$\alpha(st)=\alpha (s)$ and $\omega(st)=\omega (t)$.
For a graph $\Gamma$, the \emph{free semigroupoid} $\Gamma
^+$ on $\Gamma $ has as vertex-set $V(\Gamma )$ and as edges the
non-empty  paths on $\Gamma $.

Every semigroup $S$ may be  viewed as a semigroupoid by
taking the set of edges $S$ with both ends at an added vertex.
Conversely, for a semigroupoid~$S$ and a vertex $v$ of
$S$, the set $S(v)$ of all \emph{loops} at vertex $v$ constitutes
an semigroup called the \emph{local semigroup} of $S$
at~$v$.

A \emph{pseudovariety of  semigroupoids} is a class of
finite  semigroupoids closed under taking  divisors of
semigroupoids, and finitary direct products. The
pseudovariety of all finite  semigroupoids  is denoted by
$\pv{Sd}$.

From hereon, we assume that all semigroupoids have a finite number of vertices.
A \emph{compact semigroupoid} $S$ is a semigroupoid endowed with
a compact topology on $E(S)$ and the discrete topology
in the finite set $V(S)$, with respect to which the partial operations $\alpha$,
$\omega$, and edge multiplication are continuous (see~\cite{Almeida&ACosta:2007a} for delicate questions related with infinite-vertex
   semigroupoids).
Finite semigroupoids
equipped with the discrete topology become compact semigroupoids.
Let $\pv V$ be a pseudovariety of semigroupoids. A compact semigroupoid is \emph{pro-$\pv V$} if it is compact and every pair of distinct coterminal edges $u$ and $v$ can be separated by a continuous semigroupoid homomorphism into a semigroupoid of $\pv V$.
 For a finite graph $\Gamma $, a compact
semigroupoid $S$ is $\Gamma $-\emph{generated} if there
is a graph homomorphism $\varphi \colon \Gamma \to S$ such
that the subgraph of $S$ generated
by $\varphi (\Gamma )$ is dense.
We denote by $\Om \Gamma V$ the free pro-$\pv V$ $\Gamma $-generated
semigroupoid.
The semigroupoid  $\Om \Gamma V$ has
the usual universal property.



For each finite graph $A$,
   the free semigroupoid $A^+$ generated by $A$
   is dense in $\Om A{Sd}$.
   The edges of $A^+$
   are the
   nonempty paths on $A$, whence the name \emph{pseudopath}
   for the edges of $\Om A{Sd}$.
A generalization of Reiterman's Theorem  states that the
pseudovarieties of  semigroupoids  are the classes of finite semigroupoids
defined by pseudoidentities, that is formal identities between
\emph{pseudopaths}
 of finitely generated free pro-$\pv {Sd}$ semigroupoids (\cite{Jones:1996,Almeida&Weil:1996}).
 
 We are mostly interested in two kinds of semigroupoid pseudovarieties
 induced by a semigroup pseudovariety $\pv V$:
 the pseudovariety $g\pv V$ of semigroupoids generated
 by~$\pv V$ (the~\emph{global of $\pv V$}),
 and the pseudovariety $\ell\pv V$ of semigroupoids whose local semigroups
belong to $\pv V$. One has $g\pv V\subseteq\ell\pv V$, and if the equality holds, then the pseudovariety $\pv V$ is said to be \emph{local}.


 Let $A$ be a finite graph. In general a pseudoidentity
 between edges $\pi$ and $\rho$ of $\Om A{Sd}$
 is denoted $(A;\pi=\rho)$,
 because if $A$ is a subgraph of $B$,
 then a finite semigroupoid satisfying $(B;\pi=\rho)$
 may not satisfy $(A;\pi=\rho)$, see \cite[pages 100 and 101]{Rhodes&Steinberg:2009qt}.
 However, if $\pv V$ is a semigroup pseudovariety, then  we can write a pseudoidentity $(A;\pi=\rho)$ satisfied by $\Om A{}g\pv V$
 simply by $(\pi=\rho)$, because in that case there is no dependence
 on $A$~\cite[Theorem 2.5.15]{Rhodes&Steinberg:2009qt}.

 \subsection{Semigroupoid versions of previous definitions and results}

 Several basic concepts for semigroups carry on
 to semigroupoids without substantial modifications.
 For example, in a semigroupoid $S$
 one may consider the Green relations
 between edges, an edge $s$ is regular
 if $s=sxs$ for some edge $x$ of~$S$, etc.
 Definitions~\ref{defi:factorizations-levels}
 and~\ref{defi:organizing-pseudovariety}
 also carry on to (pseudovarieties of) semigroupoids with no real modifications: just replace pseudowords by pseudopaths,
    and words by paths.
 For example, an edge $\pi$ of $\Om A{Sd}$ is
 \emph{$\pv V$-regular} if its canonical projection $[\pi]_{\pv V}$
 in $\Om AV$ is regular, where $\pv V$ is a semigroupoid pseudovariety.
 Next are the semigroupoid versions of
 Propositions~\ref{p:reduction-of-pseudoword}
 and Corollary~\ref{c:basis-of-nice-pseudoidentities},
 for which entirely analogous proofs hold.

              \begin{Prop}\label{p:reduction-of-pseudoword-sgpoid}
        Let $\pv V$ be a pseudovariety of semigroupoids and let $A$ a finite graph.
        Suppose that $\pi$
        is an edge of $\Om A{Sd}$
        with a factorization multiregularly organized in $\pv W$, where
        $\pv W$ is a semigroupoid pseudovariety.
        Then there is an edge $\pi'$ in $\Om A{Sd}$ and a set $\Upsilon_\pi$
        of pseudoidentities over~$A$
        satisfying the following conditions:
        \begin{enumerate}
        \item
          $\pv V\subseteq\op\Upsilon_\pi\cl\subseteq \op (A;\pi=\pi')\cl$;\label{item:reduction-of-pseudoword-sgpoid-1}
        \item $\pi'$ has a factorization multiregularly organized in $\pv W$ and with short
          \mbox{$\pv V$-breaks};\label{item:reduction-of-pseudoword-sgpoid-3}
        \item if $(A;\varpi=\varrho)$ belongs to $\Upsilon_\pi$,
          then $\varpi$ and $\varrho$\label{item:reduction-of-pseudoword-sgpoid-4}
          are  $\pv W$-multiregulars.\qed
        \end{enumerate}
      \end{Prop}



      \begin{Cor}\label{c:basis-of-nice-pseudoidentities-sgpoid}
                Let $\pv V$ and $\pv W$ be pseudovarieties
        of semigroupoids, with $\pv W$ being an organizing
      pseudovariety.
      Then $\pv V$ has a basis of pseudoidentities with factorizations multiregularly organized in
      $\pv W$ and with short $\pv V$-breaks.\qed
    \end{Cor}


 \subsection{The interval $[\pv {Sl},\pv {DS}\cap\pv{RS}]$}

 We denote by $\pv{RS}$ the class of finite semigroups $S$
 whose set of regular elements is a subsemigroup of $S$.
 Although $\pv {RS}$ is not a  pseudovariety,
  the class $\pv {DS}\cap\pv {RS}$ is
  a semigroup pseudovariety,
  with
  \begin{equation}\label{eq:basis-DSRS}
      \pv {RS}\cap \pv{DS}=
      \pv{DS}\cap\op x^{\omega+1}y^{\omega+1}=(x^{\omega+1}y^{\omega+1})^{\omega+1}\cl,
  \end{equation}
    as the regular elements of a semigroup
    of $\pv {DS}$ are its group elements.
    Note also that \eqref{eq:basis-DSRS}
    yields the following corollary of Proposition~\ref{p:equational-sufficient-cond-for-closure}.

  \begin{Cor}\label{c:DS-RS-is-closed-for-bideterministic-product}
    The pseudovariety
    $\pv {DS}\cap\pv {RS}$ is closed under bideterministic product.\qed
  \end{Cor}


  Denote by $Reg(S)$ the subgraph,
   of the semigroupoid $S$, formed by the regular edges of $S$.
  Let $\pv{RSd}$ be the class of finite semigroupoids
  $S$ such that $Reg(S)$
  is a subsemigroupoid of $S$. Here $\ell{\pv{RS}}$ is the class of
  finite semigroupoids whose local semigroups belong to $\pv {RS}$.
  

  \begin{Prop}\label{p:local-regular-pseudopaths}
   The equality $\ell{\pv{RS}}=\pv {RSd}$ holds.
  \end{Prop}

  \begin{proof}
    Clearly, for every semigroupoid $S$,
    and every vertex $v$ of $S$,
    an element of the local semigroup $S_v$
    of $S$ at $v$ is regular in $S_v$ if and only if
    it is regular in $S$.
    Therefore, the inclusion
    $\pv {RSd}\subseteq \ell{\pv {RS}}$ is immediate.

    Conversely, let $S\in \ell{\pv {RS}}$,
    and let $s,t$ be consecutive regular edges of $S$.
    Take edges $x$ and $y$
    such that $s=sxs$ and $t=tyt$.
    Note that $xs$ and $ty$ are idempotents rooted at
    the vertex $v=\omega s$.
    Since $xs$ and $ty$ are regular elements
    of the local semigroup $S_v$ at $v$, we know that
    $xs\cdot ty=xs\cdot ty\cdot z\cdot xs\cdot ty$ for some
    loop $z$ belonging to $S_v$.
    We then have
    \begin{equation*}
      st=s\cdot xs\cdot ty\cdot t
      =s\cdot xs\cdot ty\cdot z\cdot xs\cdot ty\cdot t=sxs\cdot t\cdot yzx\cdot s\cdot tyt
      =st\cdot yzx\cdot st,
    \end{equation*}
    thus showing that the edge $st$ is regular in $S$.
  \end{proof}

  \begin{Cor}\label{c:local-regular-pseudopaths}
    For any finite graph $A$,
    every product of regular edges of\/ $\Om A{\ell(\pv {DS}\cap\pv {RS})}$ is
    a regular edge of $\Om A{\ell(\pv {DS}\cap\pv {RS})}$.\qed
  \end{Cor}

  \subsection{Honest pseudovarieties}

  Recall that a semigroupoid homomorphism
  is \emph{faithful} if it maps
  distinct coterminal edges to distinct coterminal edges.  
  
  \begin{Prop}[{\cite{Almeida:1996c}}]\label{p:fidelity}
    If $\pv V$ is a pseudovariety of
 semigroups, then the unique continuous semigroupoid homomorphism
 from $\Om A {}{g \pv V}$ onto $\Om {E(A)}V$
 that maps $[a]_{g\pv V}$ to $[a]_{\pv V}$, for every $a\in E(A)$,
 is a
 faithful homomorphism.
 \end{Prop}

The next proposition brings nothing new, but we did not find a
direct reference for it. The \emph{content} $c(\pi)$
of a pseudopath $\pi $ of $\Om A{Sd}$ is the subgraph $X\cup \alpha (X)\cup \omega (X)$ of $A$ where $X\subseteq E(A)$ satisfies
$c([\pi ]_{\pv S} )=\{[a]_{\pv S}\mid a\in X\}$.
For pseudopaths $\pi,\rho $, we write $c(\rho )\subseteq c(\pi )$
when $c( \rho )$ is a subgraph of~$c(\pi )$.

  \begin{Prop}\label{p:absorption-of-paths-by-regulars}
    Let $\pv V$ be a pseudovariety of semigroupoids in the interval
    $[g\pv {Sl},g\pv {DS}]$.
    Let $A$ be a finite graph.
    Suppose that the edge $\pi$ of\/ $\Om A{Sd}$ is regular in~$\pv V$,
    and let $\rho$ be an edge of $\Om A{Sd}$
    such that $\omega (\pi )=\alpha (\rho )$.
    Then we have $\pv V\models \pi \rho\mathrel{\R}\pi$
    if and only if $c(\rho)\subseteq c(\pi)$.
  \end{Prop}

  \begin{proof}
    If $\pv V\models \pi \rho\mathrel{\R}\pi$,
    then $\pv {Sl}\models \pi \rho\mathrel{\R}\pi$ holds because
    $\pv V$ contains $g\pv{Sl}$,
    whence $c(\rho)\subseteq c(\pi)$.
    Conversely, suppose that $c(\rho)\subseteq c(\pi)$.
    As $\pi$ is $\pv V$-regular,
    there is $z\in\Om A{Sd}$ such that $\pv V\models\pi=\pi z\pi$ (which implies
    that $[z\pi]_{\pv V}$ is idempotent),
    and the graph $c(\pi)$
    is strongly connected.
    Therefore, and since $c(\rho)\subseteq c(\pi)$,
    there is a path $v$ in $c(\pi)$
    such that $\alpha(v)=\omega(\rho)$
    and $\omega(v)=\alpha(z)$.
    We may then consider
    the idempotent loop $s=((z\pi)^\omega \rho v(z\pi)^\omega)^\omega$ of $\Om A{Sd}$.
    Because $c(\rho)\subseteq c(\pi)$,
    we also have $c(\pi)=c(s)=c(z\pi)$
    and $\pv {DS}\models s=(z\pi)^\omega$.
    Applying Proposition~\ref{p:fidelity},
    we get $g\pv {DS}\models s=(z\pi)^\omega$,
    thus $\pv V\models s=z\pi$.
    Therefore,
    $\pv V\models \pi=\pi s=\pi \rho v z\pi(\rho v z\pi )^{\omega -1}$,
    establishing $\pv V\models \pi\mathrel{\R} \pi \rho$.
  \end{proof}

         \begin{Defi}\label{defi:honest-pseudovarieties}
         A pseudovariety of semigroups $\pv V$ is \emph{honest}
         if, for every finite graph $A$, $a\in A$ and $\pi\in\Om A{Sd}$
         such that $\pi$ is a product of
         regular elements of $\Om A {}{g\pv V}$,
         the following conditions hold:
         \begin{enumerate}
         \item when $\omega \pi=\alpha a$, one has
           $\pv V\models\pi a\mathrel{\mathcal R}\pi
           \implies g\pv V\models\pi a\mathrel{\mathcal R}\pi$;
         \item when $\omega a=\alpha \pi$,
           one has $\pv V\models a \pi \mathrel{\mathcal L}\pi
           \implies g\pv V\models a \pi\mathrel{\mathcal L}\pi$.
         \end{enumerate}
       \end{Defi}

       \begin{Prop}\label{p:examples-of-honest-pseudovarieties}
         The semigroup pseudovarieties
         in the intervals
         $[\pv {J},\pv {DS}]$
         and
         $[\pv {Sl},\pv {DS}\cap\pv {RS}]$
         are honest.
       \end{Prop}

       \begin{proof}
         Suppose that $\pv V\in [\pv {J},\pv {DS}]$.
         Let $\pi$ and $a$ be as in Definition~\ref{defi:honest-pseudovarieties}, with $\omega \pi=\alpha a$
         (the case $\alpha \pi=\omega a$ is symmetric).
         By Lemma~\ref{l:J-reduced-are-all},
         the projection of $\pi$ in $\Om {E(A)}S$,
         still denoted $\pi$,
         is a $\pv J$-reduced multiregular
         $\pi_1\cdots\pi_n$
         in $\Om {E(A)}S$.
         Applying Lemma~\ref{l:letters-make-products-of-regulars-fall},
         from $\pv V\models \pi\mathrel{\R}\pi a$ we get $a\in c(\pi_n)$.
         It follows from Proposition~\ref{p:absorption-of-paths-by-regulars}
         that $g\pv V\models \pi_na\mathrel{\R}\pi_n$,
         whence $g\pv V\models \pi a\mathrel{\R}\pi$.

         Finally, if $\pv V\in [\pv {Sl},\pv {DS}\cap\pv {RS}]$,
         then a product of regular elements of
         $\Om A{}g\pv V$ is
         a regular element of $\Om A{}g\pv V$, by~Proposition~\ref{p:local-regular-pseudopaths}.
         Therefore,
         $\pv V\models\pi a\mathrel{\mathcal R}\pi$
         implies
         $g\pv V\models\pi a\mathrel{\mathcal R}\pi$
         by Proposition~\ref{p:absorption-of-paths-by-regulars}
         also in this case.
       \end{proof}

       We leave open the problem of
       identifying all the honest pseudovarieties.

    \begin{Prop}\label{p:sgpoid-basis-of-products-of-regular-elements}
    Let $\pv V$ be a honest pseudovariety of semigroups
    that is closed under bideterministic product.
    Suppose $\pv W$ is a semigroup pseudovariety
    containing $\pv V$ and such that $g\pv W$
    is an organizing pseudovariety.
    Then $g\pv V$ is multiregularly based in $g\pv W$.
  \end{Prop}

  \begin{proof}
    By Corollary~\ref{c:basis-of-nice-pseudoidentities-sgpoid},
    the global $g\pv V$ has a basis $\Sigma$
    formed by edge pseudoidentities
    $(\pi=\rho)$
    with pseudopath factorizations
    \begin{equation*}
      \pi= u_0\cdot\pi_1\cdot u_1\cdot\pi_2\cdots \cdot u_{n-1}
      \cdot \pi_n\cdot u_n
       \quad\text{and}\quad
      \rho= v_0\cdot\rho_1\cdot v_1\cdot\rho_2\cdots \cdot v_{m-1}
      \cdot \rho_m\cdot v_m
    \end{equation*}
    that are both organized in $g\pv W$ and with short $g\pv V$-breaks.
    Projecting in $\Om {E(A)}V$
    and seeing these factorizations
    as pseudowords factorizations,
    we see that they are multiregularly organized in $\pv W$
    and with short $\pv V$-breaks, the latter property holding because
    $\pv V$ is honest.
    Combining with $\pv V\models \pi=\rho$
    and Theorem~\ref{t:cut-nice-factorizations},
    we get that $n=m$, $u_i=v_i$
    and $\pv V\models \pi_j=\rho_j$
 for every $i\in\{0,1,\ldots,n\}$
 and $j\in \{1,\ldots,n\}$.
 In particular, we have the equalities $\alpha(\pi_j)=\omega(u_{j-1})=\omega(v_{j-1})=\alpha(\rho_j)$
 and $\omega(\pi_j)=\alpha(u_{j})=\alpha(v_{j})=\omega(\rho_j)$.
 Therefore, by Proposition~\ref{p:fidelity}, we have
 $g\pv V\models \pi_j=\rho_j$, for every $j\in \{1,\ldots,n\}$.

   As the integer $n$ depends on $(\pi=\rho)$, we denote it by $n_{(\pi=\rho)}$. For each $(\pi=\rho)\in\Sigma$,
   consider the set of pseudopath pseudoidentities
   defined by $\Gamma_{(\pi=\rho)}=\{(\pi_j=\rho_j)\mid 1\leq j\leq n_{(\pi=\rho)}\}$.
    Take the union
    $\Gamma=\bigcup_{(\pi=\rho)\in\Sigma}\Gamma_{(\pi=\rho)}$
    Then, in view of the conclusion at which we arrived in the previous
    paragraph,  we have $g\pv V=\op\Gamma\cl$, an equality whose
    detailed justification follows
    exactly the same argument as in the last lines of the proof of Corollary~\ref{c:basis-of-products-of-regular-elements}.
  \end{proof}

  \begin{Cor}\label{c:gV-is-multiregularly-based}
    Suppose that $\pv V$ is a pseudovariety
    closed under bideterministic product
    and belonging to one of the intervals $[\pv {Sl},\pv {DS}\cap \pv{RS}]$
    or $[\pv {J},\pv {DS}]$.
    Then $g\pv V$ is multiregularly based.
  \end{Cor}
  
  \begin{proof}
    The pseudovariety $g\pv {DS}$ is organizing
    by~\cite[Proposition 4.2]{Almeida&Escada:2003}.
    Therefore,  by Propositions~\ref{p:examples-of-honest-pseudovarieties}
    and~\ref{p:sgpoid-basis-of-products-of-regular-elements},
    we know that $g\pv V$ is multiregularly based in $g\pv {DS}$.
    For each 
    $g\pv {DS}$-multiregular pseudopath
    $\pi=\pi_1\cdots\pi_n$,
    take for each
    $g\pv {DS}$-regular pseudopath
    $\pi_i$ a pseudopath
    $z_i$
    such that
    $g\pv {DS}\models \pi_i=\pi_i(z_i\pi_i)^\omega$,
    and let $\pi'=\pi_1(z_1\pi_1)^\omega\cdots \pi_n(z_n\pi_n)^\omega$.
    Note that $\pi'$ is an $\pv {Sd}$-multiregular pseudopath.
  Therefore, if $\Sigma$ is a basis for
  $g\pv V$ formed by pseudoidentities
  between $g\pv {DS}$-multiregular pseudopaths,
  then
  \begin{equation*}
    \Sigma'=\{(\pi'=\rho')\mid (\pi=\rho)\in\Sigma\}
    \cup
    \{((xy)^\omega(yx)^\omega(xy)^\omega)^\omega=(xy)^\omega\}
  \end{equation*}
  is a basis for $g\pv V$,
  comprised solely by multiregular pseudopaths.
  \end{proof}
  


 \section{Preservation of locality inside
 the interval $[\pv {Sl},\pv {DS}\cap\pv {RS}]$}

A pseudovariety of semigroups is \emph{monoidal}
if it is of the form $\pv V_{\pv S}$, for some
pseudovariety of monoids~$\pv V$.

\begin{Thm}[\cite{Costa&Escada:2013}]\label{t:locality-of-kmv}
  Let $\pv V$ be a monoidal pseudovariety of semigroups.
  If\/ $\pv V$ is local then the pseudovarieties of semigroups
  $\pv K\malcev \pv V$,
  $\pv D\malcev \pv V$ and $\Lo I\malcev \pv V$
  are also local monoidal pseudovarieties of semigroups.
\end{Thm}

In this section we show that an analog of Theorem~\ref{t:locality-of-kmv}
is valid for the operator $\pv V\mapsto \overline{\pv V}$
when restricted to the interval $[\pv {Sl},\pv {DS}\cap \pv{RS}]$.
It already follows from~Corollary~\ref{c:overline-commutes}
that $\overline{\pv V}$ is
monoidal when $\pv V$ is monoidal.
To proceed, we need the following lemma.

  \begin{Lemma}\label{l:local-up-N-malcev}
    Let $\pv V$ be a
    pseudovariety of
    semigroups such that $\pv K\malcev \pv V$
    and $\pv D\malcev \pv V$ are local pseudovarieties.
    If the pseudopaths~$\pi$ and~$\rho$
    are loops such that
    $g\pv V\models \pi=\rho=\pi^2$,
    then we have
    $
    \ell(\pv N\malcev\pv V)\models \pi^\omega=\pi^\omega\rho=\rho\pi^\omega=\rho^\omega.
    $
  \end{Lemma}

  \begin{proof}
    By the easy part of the Pin-Weil basis theorem
    for Mal'cev products~\cite{Pin&Weil:1996a},
    and by Proposition~\ref{p:fidelity},  we have
      $g(\pv K\malcev \pv V)\models \pi^\omega\rho=\pi^\omega$
      and
      $g(\pv D\malcev \pv V)\models \rho\pi^\omega=\pi^\omega$.
    Since
    the intersection
    $\pv K\malcev \pv V\cap\pv D\malcev \pv V$
    contains (it is actually equal to) the pseudovariety $\pv N\malcev \pv V$,
    and since $\pv K\malcev \pv V$ and $\pv D\malcev \pv V$
    are local by hypothesis, it follows in particular
    that
      $\ell(\pv N\malcev\pv V)\models \pi^\omega\rho=\pi^\omega =\rho\pi^\omega$.
    The latter implies
    \begin{equation}\label{eq:local-up-N-malcev-3}
      \ell(\pv N\malcev\pv V)\models \pi^\omega\rho^\omega=\pi^\omega
      =\rho^\omega\pi^\omega.
    \end{equation}
    On the other hand,
    the hypothesis $g\pv V\models \pi=\rho=\pi^2$ is clearly equivalent to
    $g\pv V\models \pi=\rho=\rho^2$, and so we can
    interchange the roles of $\pi$ and $\rho$
    in~\eqref{eq:local-up-N-malcev-3},
    which altogether implies
    $\ell(\pv N\malcev\pv V)\models \pi^\omega=\rho^\omega\pi^\omega=\rho^\omega$.
  \end{proof}

  The next proposition
  is the key to obtain the result we search for.
  
    \begin{Prop}\label{p:technical-condition-on-locality}
      Let $\pv V$ be a pseudovariety of semigroups.
      If  $\pv K\malcev \pv V$
      and $\pv D\malcev \pv V$ are local,
      and if $g\pv V $
      has a basis of pseudoidentities
      between pseudopaths that are regular in
      $\ell \pv V$, then $\pv V$ is local.
  \end{Prop}

  \begin{proof}
    Let $(\pi=\rho )$ be a pseudoidentity
    satisfied by $g\pv V$
    such that
    $\pi$ and $\rho$
    are regular in $\ell \pv V$.
    The proof amounts to show that $\ell\pv V\models \pi=\rho$
    holds.

    We may take a pseudopath~$x$ such that
    \begin{equation}\label{eq:locality-DS-RS-1}
      \ell\pv V\models \pi=\pi x\pi
       \quad\text{and}
       \quad
      \ell\pv V\models x=x\pi x,
   \end{equation}
   and a pseudopath $y$ such that
    \begin{equation*}
      \ell\pv V\models \rho=\rho y\rho
       \quad\text{and}
       \quad
      \ell\pv V\models y=y\rho y.
    \end{equation*}
    As $g\pv V\models \pi=\rho$, we
    have $g\pv V\models \pi x=\rho x$.
    Applying Lemma~\ref{l:local-up-N-malcev} to the latter (note that $\ell\pv V\subseteq \ell(\pv N\malcev\pv V)$)
    and observing that $\pi x$ is idempotent in $\ell\pv V$,
    we obtain
    \begin{equation}\label{eq:locality-DS-RS-3}
    \ell\pv V\models \pi x=\pi x\cdot \rho x=\rho x\cdot \pi x.
  \end{equation}
  In the pseudovariety $\ell \pv V$,
  the pseudopath $\rho y$ is an idempotent
  which is a prefix of $\rho$, and so, in view
  of~\eqref{eq:locality-DS-RS-3}, a prefix of $\pi$ also.
  Therefore, we have
      \begin{equation*}
      \ell \pv V\models \pi=\rho y \pi.
    \end{equation*}
    Dually, we also have
    \begin{equation*}
      \ell \pv V\models \pi=\pi y\rho.
    \end{equation*}
    It follows from these facts that the pseudopath
    $x'=y\rho \cdot x\cdot\rho y$
    satisfies
    in $\ell \pv V$ the pseudoidentities
    $\pi=\pi x'\pi$ and $x'=x'\pi x'$.
    Moreover,
    in $\ell \pv V$ we also have $x'=y\rho\cdot x'\cdot \rho y$.
    Therefore, if necessary replacing $x$ by $x'$, we may suppose,
    as we do from hereon, that
    $x$ not only satisfies~\eqref{eq:locality-DS-RS-1}
    and~\eqref{eq:locality-DS-RS-3}, but also
    $\ell\pv V\models x=y\rho\cdot x\cdot \rho y$.
    In particular, $g\pv V\models x=y\rho x\rho y
    =y\pi x\pi y=y\pi y=y\rho y=y$ holds,
    whence $g\pv V\models x\pi=y\rho$.
    Then, as both $x\pi$ and $y\rho$ are idempotents
    in $\ell \pv V$,
    applying Lemma~\ref{l:local-up-N-malcev} we obtain
    $\ell \pv V\models x\pi=y\rho$. It follows
    that $\ell \pv V\models \rho
      =\rho y\rho
      =\rho x\pi=\rho x(\pi x\pi)=(\rho x\pi x)\pi$.
      Replacing in the last term of this chain of pseudoidentities
      $\rho x\pi x$ by $\pi x$ by means of~\eqref{eq:locality-DS-RS-3},
      we get
      $\ell \pv V\models \rho =\pi x\pi =\pi$,
    concluding the proof.
  \end{proof}


  We are now ready to deduce the main result of this section.
  
  \begin{Thm}\label{t:locality-DS-RS}
    Let $\pv V$ be a monoidal pseudovariety of semigroups
    in the interval $[\pv {Sl},\pv {DS}\cap \pv{RS}]$.
    If\/ $\pv V$ is local, then
    $\overline{\pv V}$ is local.
  \end{Thm}

  \begin{proof}
    By Theorem~\ref{t:locality-of-kmv},
    both $\pv K\malcev \pv V$
    and $\pv D\malcev \pv V$ are local. By Theorem~\ref{t:characterization-of-left-right-deterministic-product.}, 
    $\pv K\malcev \pv V$ and $\pv D\malcev \pv V$ are closed under bideterministic product, so 
    $$\overline{\pv V}\subseteq \overline{\pv K\malcev \pv V}=\pv K\malcev \pv V \,\text{ and }\,
    \pv K\malcev \overline{\pv V}\subseteq \pv K\malcev (\pv K\malcev \pv V)=\pv K\malcev \pv V.$$
    Then we have  
  the equality $\pv K\malcev \pv V=\pv K\malcev \overline{\pv V}$ and, by duallity, 
    $\pv D\malcev \pv V=\pv D\malcev \overline{\pv V}$.

    By Corollary~\ref{c:gV-is-multiregularly-based}, $g\overline{\pv V}$
    has a basis of pseudoidentities
    between multiregular pseudopaths.
    On the other, note that $\overline{\pv V}\subseteq \pv {DS}\cap\pv {RS}$,
    as $\pv {DS}\cap\pv {RS}$
    is closed under bideterministic product.
    In view of Proposition~\ref{p:local-regular-pseudopaths},
    it follows that
    $g\overline{\pv V}$
    has a basis of pseudoidentities
    between pseudopaths that are regular in
    $\ell(\pv {DS}\cap \pv{RS})$,
    and thus in
    $\ell\overline{\pv V}$.
    
    Altogether, the result then follows directly from Proposition~\ref{p:technical-condition-on-locality}.
  \end{proof}

  Let $\pv G$ be the pseudovariety of finite groups
  and let $\pv {Ab}$ be the pseudovariety of
  finite Abelian groups.
  It is proved in~\cite{Kadourek:2008}
  that $\pv {DG}$ is local,
  but not $\pv {D{Ab}}$.
  In contrast, we have the following.

  \begin{Cor}\label{c:application-examples-of-local-pseudovarieties}
    For every pseudovariety $\pv H$ of finite groups,
    the pseudovariety \mbox{$\pv {DH}\cap\pv {ECom}$}
    is local.
  \end{Cor}

  \begin{proof}
    It is shown in~\cite{Jones&Szendrei:1992}
    that
    $\pv {Sl}\vee \pv H$
    is local.
    Applying Theorem~\ref{t:locality-DS-RS}, we deduce
    that $\pv {DH}\cap\pv {ECom}$
    is local (cf.~ Example~\ref{eg:closure-of-Sl-and-others}).
  \end{proof} 
  
  Taking
  into account that every pseudovariety of bands
  is local~\cite{Jones&Szendrei:1992},
  or that $\pv {CR}\cap \pv H$ is local for every
  pseudovariety $\pv H$ of groups~\cite{Jones:1993},
  one finds in
  Example~\ref{eg:closure-of-CR-and-others}
  other local pseudovarieties
  resulting from Theorem~\ref{t:locality-DS-RS}.



\section{Interplay with the semidirect product with $\pv D$}
\label{sec:interpl-with-semid}

Let $\alpha$ be a unary operator  on a sublattice of semigroup pseudovarieties containing $\pv {Sl}$,
mapping monoidal pseudovarieties to monoidal pseudovarieties.
A monoidal pseudovariety of semigroups $\pv V$ containing
$\pv {Sl}$ is local if and only if $\Lo V\subseteq \pv V*\pv D$, if and only if
$\Lo V=\pv V*\pv D$~\cite{Tilson:1987}.
Therefore,  $\alpha(\pv V)$ is local
if $\pv V$ is local and the following inclusions hold:
\begin{equation*}
  \Lo (\alpha(\pv V))\subseteq \alpha(\Lo V)\qquad\text{and}
  \qquad \alpha(\pv V\ast\pv D)\subseteq \alpha(\pv V)*\pv D.
\end{equation*}
Theorem~\ref{t:locality-of-kmv}
was proved in~\cite{Costa&Escada:2013}
using this strategy.
Indeed, these inclusions hold
for each pseudovariety
of semigroups $\pv V$ containing $\pv {Sl}$
when $\alpha$ is one of the operators
$\pv V\mapsto \pv {K}\malcev\pv V$
or $\pv V\mapsto \pv {D}\malcev\pv V$.
Other operators were considered in~\cite{Costa&Escada:2013}.
As observed in~\cite{Costa&Escada:2013},
for the operator $\pv V\mapsto \pv {N}\malcev\pv V$
the first inclusion holds when $\pv V\supseteq \pv {Sl}$,
while the second fails for $\pv V=\pv {Sl}$.
In this section we verify that the strategy used in~\cite{Costa&Escada:2013}
does not work for the operator $\pv V\mapsto\overline{\pv V}$.
While
$\overline{\pv V\ast\pv D}\subseteq\overline{\pv V}\ast\pv D$
holds for many subpseudovarieties
of $\pv{DS}$ (cf.~Corollary~\ref{c:V-closed-bd-VD-also}), including $\pv {Sl}$,
we have $\Lo{\overline{\pv V}}\nsubseteq\overline{\Lo V}$
when $\pv V=\pv {Sl}$.
The latter follows from the next proposition,
in view of the equality $\pv J\cap \pv{ECom}=\overline{\pv {Sl}}$
(cf.~Example~\ref{eg:closure-of-Sl-and-others}).

\begin{Prop}\label{p:comparing-overlineLSl-JECom}
  The pseudovariety $\overline{\Lo{Sl}}$
  is strictly contained in
  $\Lo{(J\cap ECom})$.
\end{Prop}

Before we proceed to the proof,
recall from~\cite[Section 10.6]{Almeida:1994a}
the so called \emph{\mbox{$k$-superposition} homomorphism}
$\Phi_k\colon\Om AS\to (\Om {A^{k+1}}S)^1$,
the unique continuous
extension of the mapping $A^+\to (A^{k+1})^*$
that sends words of length at most $k$ into the empty word and reads
the consecutive factors of length $k+1$
in every word over $A$ with length at least $k+1$.
In the proof of Proposition~\ref{p:comparing-overlineLSl-JECom}
we only need to consider the case $k=1$.

\begin{proof}[Proof of Proposition~\ref{p:comparing-overlineLSl-JECom}]
  The inclusion
  $\overline{\Lo{Sl}}\subseteq\Lo{(J\cap ECom})$
  follows immediately from
  the inclusion $\Lo{Sl}\subseteq \Lo{(J\cap ECom})$
  and
  by Corollary~\ref{c:localization-bideterministic-product}.

  Consider the alphabet $A=\{a,b\}$
  and the pseudowords
  $u=a^\omega b^\omega a^\omega$
  and
  $v=a^\omega b^\omega a^\omega b^\omega a^\omega$
  of $\Om AS$.
  Since $u$ and $v$ have the same finite factors,
  the same set of finite prefixes, and the same set of finite
  suffixes, we have
  $\Lo{Sl}\models u=v$~(cf.~\cite{Costa:1999}). As
  $u$ and $v$ are products of idempotents,
  it follows from Corollary~\ref{c:lift-pid}
  that $\overline{\Lo{Sl}}\models u=v$.


  Let $\alpha=aa$, $\beta=bb$,
  $\gamma=ab$ and $\delta=ba$.
  We have
  the factorizations
  \begin{align*}
    \Phi_1(u)&=\alpha^{\omega-1}\cdot\gamma\cdot\beta^{\omega-2}\cdot\delta\cdot\alpha^{\omega-1},\\
    \Phi_1(v)&=\alpha^{\omega-1}\cdot\gamma\cdot\beta^{\omega-2}
    \cdot \delta\cdot \alpha^{\omega-2}
   \cdot \gamma\cdot \beta^{\omega-2}\cdot \delta\cdot \alpha^{\omega-1}.
 \end{align*}
 Both factorizations are multiregularly organized
 with short ($\pv {J}\cap\pv {ECom}$)-breaks.
 As $\pv J\cap \pv{ECom}$
 is closed under bideterministic product,
 by Theorem~\ref{t:cut-nice-factorizations}
 we have
 $\pv J\cap \pv{ECom} \not\models \Phi_1(u)=\Phi_1(v)$.
 This in turn implies
 that $\pv{(J\cap ECom)}*\pv D_1\not\models u=v$
 by~\cite[Theorem 10.6.12]{Almeida:1994a}. Since  $\pv{(J\cap ECom)}*\pv D_1$
 is contained in $\Lo{(J\cap ECom})$,
 we conclude that $\Lo{(J\cap ECom})\not\models u=v$.
\end{proof}

Recall that the Brandt semigroup $B_2$ belongs to $\pv V$
if and only if $\pv V\nsubseteq \pv {DS}$.

\begin{Lemma}\label{l:V-multiregularly-based-implies-gV}
  Let $\pv V$ be a semigroup pseudovariety containing $B_2$.
  If\/ $\pv V$ is multiregularly based, then $g\pv V$ is multiregularly based.
\end{Lemma}

\begin{proof}
  In~\cite[Theorem 5.9]{Almeida&Azevedo&Teixeira:1998}
  it is stated that if $\pv V=\op u_i=v_i\mid i\in \mathcal I\cl$
  for a family $(u_i=v_i)_{i\in \mathcal I}$
  of pseudoidentities, then
  $g\pv V=\op (u_i=v_i;A_{u_i})\mid i\in \mathcal I\cl$,
  where $A_{u_i}$ is a certain graph
  canonically built from $u_i$, with the pseudowords $u_i$ and $v_i$
  seen as pseudopaths in $\Om {A_{u_i}}{Sd}$.
  Hence, if the pseudoword $w_i$ is a regular factor of $u_i$, say  $w_i=w_ix_iw_i$ for some
 pseudoword $x_i$,
  we have in particular
  $g\pv V\models (w_i=w_ix_iw_i;A_{u_i})$, and the result follows.
\end{proof}

Recall that, for a positive integer $k$,
the pseudovariety of semigroups $\pv D_k$ is defined by $yx_1\cdots x_k=x_1\cdots x_k$, and $\pv D=\bigcup _{k\geq 1}\pv D_k$.

\begin{Prop}\label{p:inheretance-of-multiregularly-based}
  If~$g\pv V$ is multiregularly based, then
    $\pv V\ast\pv D$ and $\pv V\ast\pv D_k$ are multiregularly based,
    for every $k\geq 1$.
  \end{Prop}

  \begin{proof}
    Let $\Sigma$ be a basis of $g\pv V$
    such that, for each $(u=v;A)\in\Sigma$,
    both $u$ and~$v$ are multiregular
    edges of $\Om A{Sd}$.
    For each pseudoidentity $P$ in $\Sigma$,
    let $\mathcal H_P$ be
    the set of continuous semigroupoid homomorphisms
    from $\Om A{Sd}$ onto
    free profinite semigroups $\Om nS$, with
    $n$ running the positive integers.

    By~\cite[Theorem 3.2]{Almeida&Azevedo&Teixeira:1998},
    we can pick a certain $\varepsilon_P\in\mathcal{H}_P$,
    for each pseudoidentity $P=(u=v)$ in $\Sigma$,
    in such a way that
    \begin{equation*}
      \pv V\ast\pv D=\bigcap_{P=(u=v)\in\Sigma}
      \op\varepsilon_P(u)=\varepsilon_P(v)\cl.
    \end{equation*}
    Since $u$ and $v$ are multiregular edges,
    their homomorphic images
    $\varepsilon_P(u)$
    and
    $\varepsilon_P(v)$
    are multiregular pseudowords,
    establishing the proposition for $\pv V\ast\pv D$.

    Concerning
    $\pv V\ast\pv D_k$,
    we make a proof by induction
    on~$k\geq 1$.
    Beginning with case $k=1$,
    we apply to
    $\pv V\ast\pv D_1$
    the Almeida-Weil basis
    theorem~\cite[Theorem 5.3]{Almeida&Weil:1996}, as follows.
    For each $\varepsilon\in\mathcal H_P$,
    let $\mathcal F_{P,\varepsilon}$ be the set of
    families $\pi=(\pi_q)_{q\in V(A)}$
    of elements of $(\Om nS)^1$
    for which one has
    $\pv D_1\models \varepsilon(s)=\pi_{\omega s}$
    (actually, in a literal application
    of the Almeida-Weil basis theorem one has first
    $\pv D_1\models \pi_{\alpha s}\varepsilon(s)=\pi_{\omega s}$,
    clearly equivalent
    to
    $\pv D_1\models \varepsilon(s)=\pi_{\omega s}$,
    as $\varepsilon(s)$
    is a nonempty pseudoword).
    The Almeida-Weil basis theorem gives
    \begin{equation*}
      \pv V\ast\pv D_1
      =\bigcap_{P\in\Sigma}
      \bigcap_{\varepsilon\in\mathcal H_P}
      \bigcap_{\pi\in \mathcal F_{P,\varepsilon}}
      \op \pi_{\alpha u}\varepsilon(u)=\pi_{\alpha u}\varepsilon(v)\mid P=(u=v)\cl.
    \end{equation*}
    Fixed $\pi\in\mathcal F_{P,\varepsilon}$,
    consider the family $\pi'=(\te 1{(\pi_q)})_{q\in V(A)}$.
    Note that we have
    $\pv D_1\models
    \varepsilon(s)
    =\te 1(\pi_{\omega s})$,
    whence $\pi'\in \mathcal F_{P,\varepsilon}$.
    On the other hand,
    the pseudoidentity
    $\pi_{\alpha u}\varepsilon(u)=\pi_{\alpha u}\varepsilon(v)$
    is clearly a consequence of
    the pseudoidentity
    $\te 1(\pi_{\alpha u})\varepsilon(u)
    =\te 1(\pi_{\alpha u})\varepsilon(v)$.
    We conclude that we actually have
    \begin{equation*}
      \pv V\ast\pv D_1
      =\bigcap_{P\in\Sigma}
      \bigcap_{\varepsilon\in\mathcal H_P}
      \bigcap_{\pi\in \mathcal F_{P,\varepsilon}}
      \op \te 1(\pi_{\alpha u})\varepsilon(u)
      =\te 1(\pi_{\alpha u})\varepsilon(v)\mid P=(u=v)\cl.
    \end{equation*}
    Therefore, it suffices to check that, fixed $P=(u=v)\in \Sigma$
    and $\pi\in\mathcal F_{P,\varepsilon}$, both
    $\te 1(\pi_{\alpha u})\varepsilon(u)$
    and $\te 1(\pi_{\alpha u})\varepsilon(v)$
    are multiregular.
    By hypothesis, there is a factorization $u=wu'$
    such that $w$ is a regular edge of $\Om A{Sd}$
    and $u'$ is a product of regular edges of $\Om A{Sd}$.
    Let $z\in \Om A{Sd}$ be such that
    $w=wzw$. Take $s\in E(A)$ with $z=z's$,
    for some (possibly empty) pseudopath $z'$.
    Since $\pi'\in\mathcal F_{P,\varepsilon}$,
    and noting that $\omega s=\alpha w=\alpha u$,
    we
    have $\varepsilon(s)\leq_{\mathcal L}\te 1(\pi_{\alpha u})$.
    It follows that
      $\varepsilon (w)=\varepsilon (wzw)
      \leq_{\mathcal L}\varepsilon (sw)
      \leq_{\mathcal L}\te 1(\pi_{\alpha u})\varepsilon(w)
      \leq_{\mathcal L}\varepsilon(w)$,
    showing that $\te 1(\pi_{\alpha u})\varepsilon(w)$
    is $\mathcal L$-equivalent to the regular
    pseudoword~$\varepsilon(w)$.
    Therefore,
    $\te 1(\pi_{\alpha u})\varepsilon (u)
    =\te 1(\pi_{\alpha u})\varepsilon(w)\varepsilon(u')$
    is a product of
    regular pseudowords.
    Symmetrically, as
    $\alpha u=\alpha v$,
    the pseudoword $\te 1(\pi_{\alpha u})\varepsilon (v)$
    is also~a product of
    regular pseudowords. This shows the base step of the induction.

    We proceed with the inductive step.
    If $k>1$, then $\pv D_k=\pv D_{k-1}\ast\pv D_1$~holds (cf.~\cite[Lemma 4.1]{Almeida:1994a}),
    thus $\pv V\ast\pv D_k=(\pv V\ast\pv D_{k-1})\ast\pv D_1$.
    By the induction~hypothesis,
    $\pv V\ast\pv D_{k-1}$ is multiregularly based.
    Since $B_2\in \pv {Sl}*\pv D_1$~(cf.~\cite[Section 4]{Almeida&Azevedo&Teixeira:1998}),
    it then follows
    from Lemma~\ref{l:V-multiregularly-based-implies-gV}
    that
    $g(\pv V\ast\pv {D_{k-1}})$
    is multiregularly based. This reduces to the base case,
    thus proving that $\pv V\ast\pv D_{k}$ is multiregularly based.
  \end{proof}

  \begin{Cor}\label{c:V-closed-bd-VD-also}
    Suppose that $\pv V$ is a semigroup pseudovariety
    belonging to one
    of the intervals $[\pv {Sl},\pv {DS}\cap \pv{RS}]$
    or $[\pv {J},\pv {DS}]$.
    If\/ $\pv V$
    is closed under bideterministic product,
    then so is each of the pseudovarieties $\pv V\ast\pv D$ and $\pv V\ast\pv D_k$, for every $k\geq 1$.
    Therefore, if\/ $\pv V$ belongs to one of the intervals $[\pv {Sl},\pv {DS}\cap \pv{RS}]$
    or $[\pv {J},\pv {DS}]$,
    the inclusions
    $\overline{\pv V\ast\pv D}\subseteq\overline{\pv V}*\pv D$
    and
    $\overline{\pv V\ast\pv D_k}\subseteq\overline{\pv V}*\pv D_k$
    hold.
  \end{Cor}

  \begin{proof}
    If $\pv V$ is closed under bideterministic product,
    then $g\pv V$ is
    multiregularly based by Corollary~\ref{c:gV-is-multiregularly-based}.
    Therefore, the pseudovarieties $\pv V*\pv D$
    and $\pv V*\pv D_k$
    are multiregularly based, by Proposition~\ref{p:inheretance-of-multiregularly-based},
    which implies by
    Proposition~\ref{p:equational-sufficient-cond-for-closure}
    that they are closed under bideterministic product.
  \end{proof}

  \begin{Example}\label{eg:DG}
    The pseudovariety $\pv {DG}=\op(xy)^\omega=(yx)^\omega\cl$
    is mutiregulary based,
    and therefore
    so is each of the pseudovarieties
    $\pv {DG}*\pv {D}$ and $\pv {DG}*\pv {D_k}$.
  \end{Example}

  \begin{Remark}
    The existing proof that $\pv {DG}$ is local is a \emph{tour de force}
    that does not depend on profinite methods~\cite{Kadourek:2008}.
    Note that $\pv {DG}=\pv {IE}\malcev\pv {Sl}$,
    where $\pv {IE}=\op x^\omega=y^\omega\cl$,
    so that
    $\Lo {DG}=\pv {IE}\malcev {\Lo{Sl}}$,
    as proved in~\cite[Appendix A]{Costa&Escada:2013}.
    Since $\pv {Sl}$ is local,
    one has $\Lo{Sl}=\pv {Sl}*\pv D=\pv V(B_2)*\pv D$,
    where $\pv V(B_2)$ is the semigroup pseudovariety
    generated by $B_2$.
    Moreover, thanks to~\cite[Theorem 4.2]{Costa&Escada:2013},
    the inclusion $\pv {IE}\malcev (\pv V(B_2)*\pv D)\subseteq
    (\pv {IE}\malcev \pv V(B_2))*\pv D$ holds.
    Therefore, we may hope
    for an alternative proof
    of the locality of $\pv {DG}$
    consisting in showing
    that
    \begin{equation}\label{eq:b2-dgd1}
     \pv {IE}\malcev \pv V(B_2)\subseteq\pv {DG}*\pv {D_1}.
    \end{equation}
    This inclusion
    is indeed true, but the arguments we have for its justification
    depend on the locality of $\pv {DG}$.
    In this context, Example~\ref{eg:DG}
    is relevant,
    since it reduces the problem to
    the search of a ``profinite'' proof
    of the inclusion~\eqref{eq:b2-dgd1}
    to showing that
    $\pv {IE}\malcev \pv V(B_2)\models \pi=\rho$
    whenever
    $\pi$ and $\rho$
    are 
    multiregular pseudowords
    such that $\pv {DG}*\pv {D_1}\models \pi=\rho$.
  \end{Remark}
  
  For pseudovarieties not contained in $\pv {DS}$, we have the following proposition.
  We omit the proof,
  because it is an exercise
  using the 
  ideas
  in the proof of~\cite[Lemma 3.11]{ACosta:2007a}
  and none
  of the equational techniques that are the subject of this paper.


 \begin{Prop}\label{p:VcontainsB2-closed-bd-VD-also}
    Suppose that $\pv V$ is a semigroup pseudovariety
    such that \mbox{$B_2\in \pv V$} and
    $\pv D_1\vee\pv K_1\subseteq\pv V$.
    If\/ $\pv V$
    is closed under bideterministic product,
    then so is $\pv V\ast\pv D_k$, for every $k\geq 1$.
  \end{Prop}

  We leave open the question of whether the condition
  $\pv D_1\vee\pv K_1\subseteq\pv V$ in
  Proposition~\ref{p:VcontainsB2-closed-bd-VD-also}
  is really necessary.
  As one example that it might not be, take
  the pseudovariety $\pv{ECom}$: on one hand,
  $B_2\in \pv {ECom}$
  and $\pv D_1\nsubseteq \pv {ECom}$,
  on the other hand $\pv {ECom}*\pv D_k$ is closed
  under bideterministic product for each~$k$,
  by the combination of
  Lemma~\ref{l:V-multiregularly-based-implies-gV},
  Proposition~\ref{p:inheretance-of-multiregularly-based}
  and Proposition~\ref{p:equational-sufficient-cond-for-closure}.

\bibliographystyle{spmpsci}


\end{document}